\documentclass{aic}

\usepackage[margin = 2.5cm]{geometry}

\usepackage{amsmath}
\usepackage{amsthm}
\usepackage{amssymb}
\usepackage{mathtools}
\usepackage{graphicx}
\usepackage{xcolor}
\usepackage{caption}% http://ctan.org/pkg/caption
\usepackage[square,sort,comma,numbers]{natbib} 
\usepackage{tikz} %package for adding diagram

\usepackage{setspace}
\usepackage{cleveref}
\theoremstyle{plain}

\newtheorem{theorem}{Theorem}[section]
\crefname{theorem}{Theorem}{Theorems}
\Crefname{theorem}{Theorem}{Theorems}

\newtheorem*{lemma*}{Lemma}
\newtheorem{lemma}[theorem]{Lemma}
\crefname{lemma}{Lemma}{Lemmas}
\Crefname{lemma}{Lemma}{Lemmas}

\newtheorem*{claim*}{Claim}

\crefname{claim}{Claim}{Claims}
\Crefname{claim}{Claim}{Claims}

\newtheorem{proposition}[theorem]{Proposition}
\crefname{proposition}{Proposition}{Propositions}
\Crefname{proposition}{Proposition}{Propositions}

\newtheorem{corollary}[theorem]{Corollary}
\crefname{corollary}{Corollary}{Corollaries}
\Crefname{corollary}{Corollary}{Corollaries}

\newtheorem{conjecture}[theorem]{Conjecture}
\crefname{conjecture}{Conjecture}{Conjectures}
\Crefname{conjecture}{Conjecture}{Conjectures}

\crefname{question}{Question}{Questions}
\Crefname{question}{Question}{Questions}

\crefname{observation}{Observation}{Observations}
\Crefname{observation}{Observation}{Observations}

\crefname{example}{Example}{Examples}
\Crefname{example}{Example}{Examples}

\theoremstyle{definition}

\crefname{problem}{Problem}{Problems}
\Crefname{problem}{Problem}{Problems}

\newtheorem{definition}[theorem]{Definition}
\crefname{definition}{Definition}{Definitions}
\Crefname{definition}{Definition}{Definitions}

\newtheorem{remark}[theorem]{Remark}
\usepackage{xpatch}
\xpatchcmd{\proof}{\itshape}{\normalfont\proofnamefont}{}{}
\newcommand{\proofnamefont}{}
\renewcommand{\proofnamefont}{\bfseries}

\usepackage{framed}

\newcommand{\remove}[1]{}

\newcommand{\al}{\alpha}

\newcommand{\eps}{\varepsilon}
\newcommand{\lam}{\lambda}

\DeclareMathOperator{\ex}{ex}

\newcommand{\expo}{{\rm exp}}
\newcommand{\C}{\mathcal{C}}

\newcommand{\A}{\mathcal{A}}
\newcommand{\ba}{{\bf a}}

\newcommand{\cC}{\mathcal{C}}

\newcommand{\E}{\mathbb{E}}
\newcommand{\F}{\mathcal{F}}
\newcommand{\G}{\mathcal{G}}
\newcommand{\cH}{\mathcal{H}}
\newcommand{\I}{\mathcal{I}}
\newcommand{\K}{\mathcal{K}}
\newcommand{\bk}{{\bf k}}
\newcommand{\cS}{\mathcal{S}}
\newcommand{\Pb}{\mathbb{P}}

\newcommand{\hb}{\hat{\beta}}
\newcommand{\hc}{\hat{C}}

\aicAUTHORdetails{%
  title = {Balanced Supersaturation and Tur\'an Numbers in Random Graphs}, %% please capitalize all significant words
  author = {Tao Jiang, Sean Longbrake},
  plaintextauthor = {Tao Jiang, Sean Longbrake},
  plaintexttitle = {Balanced Supersaturation and Turan Numbers in Random Graphs}, %%  title without math or LaTeX
  runningtitle = {Balanced Supersaturation}, 
  keywords = {supersaturation, extremal-graph-theory},
}   %%% END \aicAUTHORdetails

\aicEDITORdetails{%
   year={2024},
   number={3},
   received={12 March 2023},   % received date: example: 7 January 2017
   published={15 July 2024},  % published date
   doi={10.19086/aic.2024.3},      % XXX = number of paper, e.g. aic006 for paper#6
}   %%% END \aicEDITORdetails

\begin{document}

\begin{frontmatter}[classification=text]
%% EDITOR: this will force the keywords to appear right after the Abstract.
%%   If the abstract is too long and would force the keywords off the
%%   front page, please comment out % [classification=text] above
%%   This way the keywords will be floated on the bottom of the first page
%%   even though the Abstract spills over to the next page.

%%% AUTHOR: Title goes here.  This line is optional.  You must use it
%%   if title has footnote attached or requires nontrivial typesetting,
%%   e.g., inclusion of linebreaks to force nice layout.
%%\title{Short Proof of R\"odl's $n^{\log\log n}$ Bound\titlefootnote{This is a footnote to the title}} %% please capitalize all significant words

%%% AUTHOR:
%%% List all authors. If you wish, place grant acknowledgements in \thanks.
%%% In brackets include a short tag for each author.
\author[tjiang]{Tao Jiang\thanks{Research supported by National Science Foundation grant DMS-1855542.}}
\author[slong]{Sean Longbrake\thanks{ Research partially supported by National Science Foundation grant DMS-1855542.}}

%%% AUTHOR: Abstract goes here
\begin{abstract}
In a ground-breaking paper solving a conjecture of Erd\H{o}s on the number of
$n$-vertex graphs not containing a given even cycle, Morris and Saxton  \cite{MS}
made a broad conjecture on so-called balanced supersaturation property
of a bipartite graph $H$. Ferber, McKinley, and Samotij \cite{FMS} established a weaker version of this conjecture  and applied it to  derive far-reaching results on the enumeration problem of $H$-free graphs.

In this paper, we show that Morris and Saxton's conjecture holds under
a very mild assumption about $H$, which is widely believed to hold whenever $H$ contains a cycle. We then use our theorem to obtain enumeration results and general upper bounds on the Tur\'an number of a bipartite $H$ in the random graph $G(n,p)$, the latter being the first of its kind.\end{abstract}
\end{frontmatter}

%%% AUTHOR: body of paper starts here
\section{Introduction}

Given a graph $H$, we say that a graph $G$ is {\it $H$-free} if it doesn't contain $H$ as a subgraph.
For a family $\F$ of graphs, we say $G$ is {\it $\F$-free } if it doesn't contain any member of $\F$ as a subgraph.
For a given positive integer $n$ and a graph $H$, the {\it extremal number} $\ex(n,H)$ denotes the maximum number of edges in an $H$-free $n$-vertex graph. (For a family $\F$, $\ex(n,\F)$ is analogously defined for $\F$-free graphs.)
A central problem in extremal graph theory is to determine the extremal number $\ex(n,H)$
and the typical structure of $H$-free graphs. Such a study was initiated by Tur\'an \cite{turan} in the 1940s, who determined precisely
the extremal number of the complete graph. The problem of studying $\ex(n,H)$ is therefore also referred to as the {\it Tur\'an problem}.
 Erd\H{o}s and Stone \cite{ES-stone} (see also \cite{ES-1966}) determined asymptotically the extremal number $\ex(n,H)$ for any non-bipartite graph $H$,  thus leaving estimating $\ex(n,H)$ for a bipartite graph $H$ the
main remaining challenge in the field. K\H{o}v\'ari, S\'os and Tur\'an \cite{KST} showed that $\ex(n,K_{s,t})=O(n^{2-1/s})$,
where $K_{s,t}$ denotes the complete bipartite graph with part sizes $s$ and $t$. This was later shown to be asymptotically tight when $t>s!$ by Koll\'ar, R\'onyai, and Szab\'o \cite{KRS}
and when $t>(s-1)!$ by Alon, R\'onyai, and Szab\'o \cite{ARS}.  Both of these were obtained via algebraic constructions.
More recently, an innovative random algebraic approach had led to many new
tight lower bound constructions, see  \cite{BBK, bukh, BC} for instance. 
In particular, Bukh and Conlon \cite{BC} applied the method to settle a long-standing conjecture that asserts that 
for every rational number $\al$ in $[1,2)$ there is a family $\F$ of bipartite graphs for which $\ex(n,\F)=\Theta(n^\al)$.
A series of recent progresses have also been made on the related exponent conjecture for single bipartite graphs $H$ (rather than
forbidding a family $\F$), resulting in many rationals $\al\in [1,2]$ and bipartite graphs $H$ for which $\ex(n,H)=\Theta(n^\al)$,
see for instance \cite{CJ,CJL, Janzer, JJM,JMY, JQ}.  Another general result  is due to F\"uredi \cite{Furedi}, and independently
to Alon, Krivelevich, and Sudakov \cite{AKS} that asserts that $\ex(n,H)=O(n^{2-1/s})$ for every bipartite graph $H$ where
vertices in one part have degree at most $s$ (see \cite{GJN, JL} for recent generalizations of this result). 
Alon, Krivelevich, Sudakov \cite{AKS}, in addition, showed that
$\ex(n,H)=O(n^{2-1/4s})$ for every $s$-degenerate bipartite graph $H$. Despite these substantial progresses on the
bipartite Tur\'an problem, it remains to be the case that for most bipartite graphs $H$ there exist substantial gaps between
best known lower and upper bounds on $\ex(n,H)$, even for even cycles $C_{2\ell}$, where $\ell\neq 2,3,5$. For more background,
the reader is referred to the excellent survey by F\"uredi and Simonovits \cite{FS-survey}.

Erd\H{o}s, Kleitman and Rothschild \cite{EKR} introduced
the problem of counting $H$-free graphs on $n$ vertices.
They showed that there are $2^{(1+o(1)) \ex(n,K_r)}$ $K_r$-free graphs,
and furthermore that almost all triangle-free graphs are bipartite.
Erd\H{o}s, Frankl and R\"odl \cite{EFR} generalized the former result, showing that
for any non-bipartite $H$ there are $2^{(1+o(1))\ex(n,H)}$ $H$-free graphs.
Kolaitis, Pr\"omel and Rothschild \cite{KPR} generalized
the latter result,  showing that almost all $K_r$-free graphs are $(r-1)$-partite.
This was further extended by Pr\"omel and Steger \cite{PS} to all $r$-critical graphs.
Using hypergraph regularity method, Nagle, R\"odl and Schacht \cite{NRS} showed
that there are $2^{(1+o(1))\ex(n,H)}$ $H$-free $k$-uniform hypergraphs for any non-$k$-partite
$k$-uniform hypergraph $H$.
Balogh, Morris
and Samotij \cite{BMS} and Saxton and Thomason \cite{ST} reproved this result 
using the hypergraph container method.
 Balogh, Bollob\'as and Simonovits \cite{BBS1, BBS2, BBS3} proved
more precise counting and structural results for graphs.

For bipartite $H$,  Kleitman and Winston \cite{KW1982} made the first breakthrough by  showing that there are at most
$2^{(1+c) n^{3/2}}$ $C_4$-free graphs on $n$ vertices, where $c\approx 0.0819$ and
resolving a long-standing question of Erd\H{o}s.  No further progress was made until the recent significant work of
Balogh and Samotij \cite{BS1, BS2}, who showed for every $2\leq s\leq t$ that there are at most $2^{(O(n^{2-1/s}))}$ $K_{s,t}$-free
graphs on $n$ vertices. The next major breakthrough was made by Morris and Saxton \cite{MS}, who showed that the number of $C_{2\ell}$-free graphs
on $n$ vertices is at most $2^{O(n^{1+1/\ell})}$, confirming a conjecture of Erd\H{o}s. There are two key ingredients in Morris and Saxton's work. 
One is the framework of the container method  and the other is the so-called balanced supersaturation property of
$C_{2\ell}$, stated in one of their main theorems as below.

\begin{theorem} [Morris-Saxton \cite{MS}] \label{MS-cycle-supersat}
For every $\ell \geq 2$, there exist constants $C>0,\delta>0$ and $k_0\in \mathbb{N}$ such that the following
holds for every $k\geq k_0$ and every $n\in \mathbb{N}$. Given a graph $G$ with $n$ vertices and $kn^{1+1/\ell}$ edges,
there exist a collection $\cH$ of copies of $C_{2\ell}$ in $G$, satisfying:
\begin{enumerate}
\item[{\rm (a)}] $|\cH|\geq \delta k^{2\ell} n^2$, and

\item[{\rm (b)}] $d_\cH(\sigma)\leq C\cdot k^{2\ell-|\sigma|-\frac{|\sigma|-1}{\ell-1}} n^{1-1/\ell}$ for every $\sigma \subset E(G)$
with $1\leq |\sigma|\leq 2\ell-1$,
where $d_\cH(\sigma)=|\{A\in \cH: \sigma\subset A\}|$ denotes the `degree' of the set $\sigma$ in $\cH$.
\end{enumerate}
\end{theorem}

Using their container method framework and Theorem~\ref{MS-cycle-supersat}, Morris and Saxton
~\cite{MS} were able to not only obtain the enumeration result on the number of $C_{2\ell}$-free graphs on $n$ vertices,
but also make further progress on the Tur\'an number of $C_{2\ell}$ in the Erd\H{o}s-R\'enyi random graph $G(n,p)$,
where as usual $G(n,p)$ denotes the random graph on $[n]$ where each pair $ij$ is included as an edge independently with
probability $p$. Given a graph $H$, let 
\[\ex(G(n,p),H):=\max\{e(G): G\subset G(n,p) \mbox{ and } G \mbox { is $H$-free} \}.\]
Note  that both $G(n,p)$ and $\ex(G(n,p),H)$ are random variables.

The problem of determining the threshold function for 
the maximum number of edges in an $H$-free subgraph of $G(n,p)$
has received much attention.
For more thorough discussion, the reader is referred to the excellent survey by R\"odl and Schacht~\cite{RS}.
The most significant work was the following breakthrough, first independently obtained by Conlon and Gowers \cite{CG}
(under the assumption that $H$ is strictly $2$-balanced, see Definition~\ref{r-density-definition}) and by Schacht \cite{Schacht},
then reproved using the container method by Balogh, Morris and Samotij \cite{BMS} and independently by Saxton and Thomason~\cite{ST}.
\begin{theorem} [Conlon-Gowers \cite{CG}, Schacht \cite{Schacht}] \label{random-main}
Let $H$ be a graph with $\Delta(H)\geq 2$ and chromatic number $\chi(H)$.
Let $m_2(H)=\max\{(e(F)-1)/(v(F)-2): F\subseteq H, v(F)\geq 3\}$.
For every $\eps>0$ there exists a constant $C>0$ such that if $p\geq Cn^{-1/m_2(H)}$, then 
\[\ex(G(n,p), H)\leq (1-\frac{1}{\chi(H)-1}+\eps) \binom{n}{2} p,\]
with high probability, as $n\to \infty$.
\end{theorem}

While Theorem~\ref{random-main} gives a satisfactory answer for non-bipartite $H$, much less is known for bipartite $H$. For $C_{2\ell}$, Haxell, Kohayakawa and {\L}uczak \cite{HKL} showed that if $p\gg n^{-1+1/(2\ell-1)}$ then
$\ex(G(n,p), C_{2\ell})\ll e(G(n,p))$, whereas if $p=o(n^{-1+1/(2\ell-1})$ then $\ex(G(n,p), C_{2\ell})=(1+o(1)) e(G(n,p))$.  For $p=\al n^{-1+1/(2\ell-1)}$ and $2\leq \al\leq n^{1/(2\ell-1)^2}$,
Kohayakawa, Kreuter and Steger \cite{KKS} obtained the tight result that with high probability, 
\[\ex(G(n,p), C_{2\ell})=\Theta \left(n^{1+1/(2\ell-1)} (\log \al)^{1/(2\ell-1)}\right).\]
 For recent work on the analogous problem for linear cycles in random $r$-uniform hypergraphs, see Mubayi and Yepremyan \cite{MY}.

Applying the container method to Theorem~\ref{MS-cycle-supersat}, Morris and Saxton obtained the following.
\begin{theorem} [Morris-Saxton \cite{MS}] \label{MS-cycle-random}
For every $\ell\geq 2$, there exists a constant $C=C(\ell)>0$ such that

\[\ex(G(n,p), C_{2\ell})\leq \begin{cases}  Cn^{1 + 1/(2\ell - 1)}(\log n)^{2} & \text{if } p\leq n^{- (\ell - 1)/(2\ell-1)}  \cdot (\log n)^{2\ell} 
\\
Cp^{1/\ell } n^{1 + 1/\ell} & \text{otherwise}
\end{cases}\]

with high probability as $n\to \infty$.
\end{theorem}

A well-known conjecture of Erd\H{o}s and Simonovits (see \cite{FS-survey}) states that $\ex(n,\{C_3,C_4,\dots, C_{2\ell}\})=\Theta(n^{1+1/\ell})$.
Morris and Saxton \cite{MS} further showed that with high probability as $n\to\infty$, $\ex(n,G(n,p))=\Omega(p^{1/\ell}n^{1+1/\ell})$
for each $\ell$ for which the Erd\H{o}s-Simonovits conjecture is true. The successful applications of Theorem~\ref{MS-cycle-supersat} 
to both the enumeration problem and the random Tur\'an problem on $C_{2\ell}$ motivated
Morris and Saxton \cite{MS} to make a general conjecture about all bipartite graphs.

\begin{conjecture} [Morris-Saxton \cite{MS}] \label{MS-conjecture}
Given a bipartite graph $H$ containing a cycle,\footnote{While the phrase "containing a cycle" does not appear in Morris-Saxton's conjecture, context makes it clear that it was intended. It is easy to see the conjecture is false for forests.} there exist constants $C>0, \eps>0$ and $k_0\in \mathbb{N}$ such that the following
holds. Let $k\geq k_0$, and suppose that $G$ is a graph on $n$ vertices with $k\cdot \ex(n,H)$ edges. Then there exists
a (non-empty) collection\footnote{For such a collection to exist, it is neccessary that $|\cH| \geq C^{-1} k^{(1 + \varepsilon)(e(H) - 1)}e(G)$, as can be seen by letting $\sigma$ be a member of $\cH$.} $\cH$ of copies of $H$ in $G$, satisfying
\[d_\cH(\sigma)\leq \frac{C\cdot |\cH|}{k^{(1+\eps)(|\sigma|-1)} e(G)} \mbox{ for every } \sigma\subset E(G)
\mbox{ with } 1\leq |\sigma| \leq e(H).\]
\end{conjecture}

An important aspect of this conjecture is that  its truth would immediately yield desired enumeration results on $H$-free graphs, 
in the following sense.
\begin{proposition} [Morris-Saxton \cite{MS}] \label{MS-proposition}
Let $H$ be a bipartite graph. If Conjecture \ref{MS-conjecture} holds for $H$, then there are at most $2^{O(\ex(n,H))}$ $H$-free 
graphs on $n$ vertices.
\end{proposition}

The work of Morris and Saxton and Conjecture~\ref{MS-conjecture} generated a lot of interest in the field.
In a recent breakthrough, Ferber, McKinley and Samotij \cite{FMS} were able to establish a weaker version of Conjecture~\ref{MS-conjecture}
and applied it to obtain very general enumeration results on $H$-free hypergraphs for all $r$-partite $r$-uniform hypergraphs ($r$-graphs in short) $H$ that satisfy a very mild assumption which is widely believed to hold for all $r$-partite $r$-graphs.
The weaker version of Conjecture~\ref{MS-conjecture} that Ferber, McKinley and Samotij~\cite{FMS} proved is a bit technical to
state here and does not seem to immediately apply to the random Tur\'an problem.
However, the enumeration results it yields are very general and significant.

\begin{definition}[$r$-density and proper $r$-density] \label{r-density-definition}
Let $r\geq 2$ be an integer. Given an $r$-graph $H$ with $v(H)\geq r+1$, we define its {\it $r$-density} to be
\[m_r(H)=\max\{\frac{e(F)-1}{v(F)-r}: F\subseteq H, v(F)>r\}.\]
	We define its {\it proper $r$-density} to be
\[m^*_r(H)=\max\{\frac{e(F)-1}{v(F)-r}: F\subsetneq H, v(F)>r\}.\]
If $m_r(H)>m^*_r(H)$, we say that $H$ is {\it strictly $r$-balanced}.
\end{definition} 
For instance, for the even cycle $C_{2\ell}$, we have $m_2(C_{2\ell})=\frac{2\ell-1}{2\ell-2}$ and
$m^*_2(C_{2\ell})=1$. Throughout the paper, we will tacitly assume that $H$ has at least two edges.
The main results of Ferber, McKinley and Samotij are 
\begin{theorem} [Ferber-McKinley-Samotij \cite{FMS}] \label{FMS-main1}
Let $H$ be an $r$-uniform hypergraph and let $\al$ and $A$ be positive constants. Suppose that 
$\al>r-\frac{1}{m_r(H)}$ and that $\ex(n,H)\leq An^\al$ for all $n$. Then there exists a constant $C$ depending 
only on $\al, A$, and $H$ such that for all $n$, there are at most $2^{Cn^\al}$ $H$-free $r$-uniform hypergraphs on $n$ vertices.
\end{theorem}

\begin{theorem} [Ferber-McKinley-Samotij \cite{FMS}] \label{FMS-main}
Let $H$ be an $r$-uniform hypergraph and assume that $\ex(n,H)\geq \eps n^{r-\frac{1}{m_r(H)}+\eps}$ for some $\eps>0$ and all $n$.
Then there exists a constant $C$ depending 
only on $\varepsilon$ and $H$ such that for all $n$, there are at most $2^{C\ex(n, H)}$ $H$-free $r$-uniform hypergraphs on $n$ vertices.
\end{theorem}

To get a sense of where the condition on $H$ in Theorem~\ref{FMS-main} came from, observe that a simple
first moment argument shows that for any $r$-uniform hypergraph $H$, $\ex(n,H)\geq \Omega(n^{r-1/m_r(H)})$ holds.
In the $2$-uniform case, it is widely believed that this simple probabilistic lower bound is not asymptotically tight for any $H$ that contains a cycle.
The $r\geq 3$ case is expected to be similar.
Ferber, McKinley and Samotij made the following conjecture.

\begin{conjecture}[Ferber-McKinley-Samotij \cite{FMS}] \label{FMS-conjecture}
Let $H$ be an arbitrary graph that is not a forest. There exists an $\eps>0$ such that
$\ex( n,H)\geq \eps n^{2-1/m_2(H)+\eps}$.
\end{conjecture}

Conjecture~\ref{FMS-conjecture} is known to hold for quite a few families of bipartite graphs,
including complete bipartite graphs, even cycles, the cube graph, and etc (see
\cite{FS-survey} and \cite{FMS} for some discussions). In particular, the work of Bukh and Conlon \cite{BC} on 
the Tur\'an exponent of a bipartite family provides a large
family of bipartite graphs $H$ for which Conjecture \ref{FMS-conjecture} holds, namely graphs $H$ that
are obtained by gluing enough copies of a so-called balanced tree at the leaves. There is also other strong
evidence that Conjecture~\ref{FMS-conjecture} should be true. Bohman and Keevash
\cite{BK} showed that if $H$ is a bipartite graph that is strictly $2$-balanced (see Defintion~\ref{r-density-definition})  then
$\ex(n,H)\geq \Omega(n^{2-\frac{1}{m_2(H)}} (\log n)^{1/(e(H)-1)})$. Bennett and Bohman \cite{BB} later generalized this result for hypergraphs (See also  \cite{FMS} for
another generalization). On the other hand, a well-known conjecture of Erd\H{o}s and Simonovits
\cite{erdos} says that for any bipartite graph $H$, there exist constants $\al\in [1,2), c_1,c_2>0$ such that
$c_1n^\al\leq \ex(n,H)\leq c_2n^\al$. Hence, if the Erd\H{o}s-Simonovits conjecture were true, then the result of Bohman and Keevash would imply Conjecture~\ref{FMS-conjecture}.
%%%%%%%%%%%%%%%%%%%%%%%%%%%%%%%%%%%%%%%%%%%%%%

\section{Main results}

As our main result of the paper, we prove Conjecture~\ref{MS-conjecture} of Morris and Saxton 
in a more explicit form under the same mild condition about $H$ assumed by Ferber, McKinley and Samotij \cite{FMS}. 

\begin{theorem}[Balanced Supersaturation] \label{main1}
Let $r$ be an integer with $r\geq 2$. Let $H$ be
an $r$-partite $r$-graph with $h$ vertices and $\ell$ edges.
Let $\al$ and $A$ be positive reals satisfying that $A\geq r^{2r}$, $\alpha>r-\frac{1}{m_r(H)}$ and that $\ex(n,H)\leq An^\al$ for all $n$,
There exist constants $k_0, C>0$ such that the following holds.
Let $G$ be an $n$-vertex $r$-graph with $m= kn^\alpha$ edges where $k\geq k_0$.
Then $G$ contains a non-empty family $\cH$ of copies of $H$ such that
, \[d_\cH(S)\leq C k^{-\lam(\al, H)(|S|-1)}\frac{|\cH|}{e(G)}, 
\mbox{ for every } S\subseteq E(G), 1\leq |S|\leq e(H),\]
where 
$\lam(\al, H)=\frac{1}{m_r(H)(r-\al)}$.
\end{theorem}
Since $\lam(\al,H)>1$, Theorem~\ref{main1} resolves Conjecture~\ref{MS-conjecture} in a more explicit form under the mild assumption that $\al>r-1/m_r(H)$.
Our general approach for establishing Theorem~\ref{main1} is inspired by the approach used by Ferber, McKinley and Samotij \cite{FMS}. However, we also added  some crucial new twists, in particular establishing a stronger supersaturation theorem for general bipartite graphs then was currently known.
Theorem~\ref{main1} allows one to retrieve Theorem~\ref{FMS-main1} and Theorem~\ref{FMS-main}. 

Theorem~\ref{main1} does not explicitly describe how dense the family $\cH$ is.
Futhermore, we generalized the balanced supersaturation result in Ferber, McKinley, Samotij \cite{FMS}  to turn
any dense family of copies of $H$ in $G$ into a balanced one that is almost as dense. In view of
that, we can obtain an even stronger version of Theorem~\ref{main1} for those $H$
that satisfy the following well-known conjecture of Erd\H{o}s and Simonovits, which roughly
says that one can expect to find asymptotically as many copies of $H$ in $G$ as  one would expect in a random graph with the same edge-density as $G$.

\begin{conjecture}[Erd\H{o}s-Simonovits] \label{ES-conjecture-supersat} 
Let $H$ be a bipartite graph with $h$ vertices and $\ell$ edges. Let $A,\al$ be positive reals satisfying that $\ex(n,H)\leq An^\al$
for all $n\in \mathbb{N}$. There exists constant $C,c$, depending on $H$ such that for all sufficiently large $n$ if
$G$ is an $n$-vertex graph with $e(G)>Cn^\al$ edges then $G$ contains at least $c[e(G)]^\ell/n^{2\ell-h}$ copies of $H$.
\end{conjecture}

Given a bipartite graph $H$, we say that $H$ is {\it Erd\H{o}s-Simonovits good} if it satisfies Conjecture~\ref{ES-conjecture-supersat}.
There are quite a few known Erd\H{o}s-Simonovits good graphs for appropriate values of $\alpha$,
for instance, even cycles \cite{FS} (see also \cite{MS}), complete bipartite graphs \cite{ES-supersat}, bipartite graphs that have a vertex complete to the other part \cite{CFS}, tree blowups \cite{GJN}, tree degenerate graphs \cite{JL}, etc. For Erd\H{o}s-Simonovits good $H$, we obtain the following stronger theorem.

\begin{theorem}[Balanced supersaturation for Erd\H{o}s-Simonovits good graphs]
 \label{main2}
Let $H$ be a bipartite graph with $h$ vertices and $\ell$ edges.
Suppose that $H$ is Erd\H{o}s-Simonovits good.
Let $\alpha$ and $A$ be positive reals satisfying that $A\geq 16$, $\alpha>2-\frac{1}{m_2(H)}$ and that $\ex(n,H)\leq An^\alpha$ for all $n$.
There exist constants $\delta_H,k_0, C>0$ such that the following holds.
Let $G$ be an $n$-vertex graph with $ kn^\alpha$ edges where $k\geq k_0$.
Then $G$ contains a family $\cH$ of copies of $H$ satisfying that
\begin{enumerate}
\item $|\cH|\geq \delta_H [e(G)]^{\ell}/n^{2\ell-h}$,
\item  $\forall S\subseteq E(G), 1\leq |S|\leq \ell$,  \[d_\cH(S)\leq C \beta^{|S|-1}\frac{|\cH|}{e(G)},\]
where 
\[\beta=\max\{ (1/k) n^{-\phi(\al, H)}, k^{-\lam^*(\al, H)}\},\]
$\phi(\al,H)=\frac{\al\ell-\al+h-2\ell}{\ell-1}$,
$\lam^*(\al,H)=\frac{1}{m^*_2(H)(2-\al)}$.
\end{enumerate}
\end{theorem}

Equipped with Theorem~\ref{main1} and Theorem~\ref{main2},
we then apply them under the framework of Morris and Saxton~\cite{MS} to obtain general bounds on $\ex(G(n,p), H)$ 
as follows.

\begin{theorem} \label{random-Turan1}
Let $H$ be a bipartite graph with $h$ vertices and $\ell$ edges. Let $A,\al$ be positive reals such that $\ex(n,H)\leq An^\al$
for every $n\in \mathbb{N}$.  There exists a constant $C=C(H)$ such that 

\[\ex(G(n,p), H)\leq \begin{cases} 
   Cn^{2-\frac{1}{m_2(H)}} & \text{if } p\leq n^{-\frac{1}{m_2(H)}},\\
Cp^{1-\frac{1}{\lam(\al,H)} }n^\al & \text{otherwise } 
\end{cases}\]
with high probability as $n\to \infty$.
\end{theorem}

\begin{theorem}\label{random-Turan2}
Let $H$ be a bipartite graph with $h$ vertices and $\ell$ edges such that $H$ is Erd\H{o}s-Simonovits good. Let $A,\al$ be positive reals such that $\ex(n,H)\leq An^\al$
for every $n\in \mathbb{N}$.  There exists a constant $C=C(H)$ such that 
\[\ex(G(n,p), H)\leq \begin{cases}  Cn^{\alpha - \phi(\al, H)}(\log(n))^{2} & \text{if } p\leq n^{-\frac{\phi(\al,H)\lam^*(\al,H)}{\lam^*(\al,H)-1} } \cdot (\log n)^{\frac{2\lam^*(\al, H)}{\lam^*(\al, H)-1}} 
\\
Cp^{1-\frac{1}{\lam^*(\al, H)} }n^\al & \text{otherwise}
\end{cases}\]
with high probability as $n\to \infty$,  where
$\phi(\al,H)=\frac{\al\ell-\al+h-2\ell}{\ell-1}$,
$\lam^*(\al,H)=\frac{1}{m^*_2(H)(2-\al)}$.
\end{theorem}

Theorem~\ref{random-Turan2} implies
Morris and Saxton's result on $\ex(G(n,p), C_{2\ell})$ (see Corollary~\ref{c2ell-random}).
To the best of our knowledge, 
Theorem~\ref{random-Turan1} and Theorem~\ref{random-Turan2} appear to be the first general results on $\ex(G(n,p),H)$ for bipartite $H$.

The rest of the paper is organized as follows. In Section~\ref{sec:balanced-supersat}, we derive our results
on balanced supersaturation. In Section~\ref{sec:random-Turan}, we apply our supersaturation results to derive general bounds
on the random Tur\'an problem.

%%%%%%%%%%%%%%%%%%%%%%%%%%%%%%%%%%%%%%%%%%%%%%%%%%%%%%

\section{Balanced supersaturation} \label{sec:balanced-supersat}

We start with a standard estimation lemma.

\begin{lemma} \label{binom-estimate}
Let $n\geq w\geq h$ be positive integers. Then $\binom{n-h}{w-h}/\binom{n}{w}\leq (w/n)^h$. Furthermore, if $w\geq h^2$
then $\binom{n-h}{w-h}/\binom{n}{w} \geq (1/2) (w/n)^h$.
\end{lemma}
\begin{proof}
We have
\[\frac{\binom{n-h}{w-h}}{\binom{n}{w}}=\frac{w(w-1)\dots (w-h+1)}{n(n-1)\cdots(n-h+1)}.\]
Hence, $\binom{n-h}{w-h}/\binom{n}{w}<(w/n)^h$. If $w\geq h^2$, then
\[\frac{\binom{n-h}{w-h}}{\binom{n}{w}}>
\frac{w(w-1)\cdots (w-h+1)}{n^h}>\frac{w^h-\binom{h}{2} w^{h-1}}{n^h} >\frac{w^h[1-\binom{h}{2} (1/w)]}{n^h} >
\frac{1}{2} \frac{w^h}{n^h}. \]
\end{proof}

The following simple lemma is folklore. We include a proof for completeness.

\begin{lemma} \label{simple-bound}
Let $r\geq 2$. Let $H$ be an $r$-graph. Let $G$ be an $n$-vertex $r$-graph with $e(G)>\ex(n,H)$. Then
$G$ contains at least $e(G)-\ex(n,H)$ different copies of $H$.
\end{lemma}
\begin{proof} Let $m$ be the number of copies of $H$ in $G$.  Then we can find a set $S$
of at most $m$ edges whose removal destroys all the copies of $H$ in $G$. Hence $e(G)-m\leq \ex(n,H)$
and thus $m\geq e(G)-ex(n,H)$.
\end{proof}

We next give a nontrivial lower bound on the number of copies of any given $r$-partite $r$-graph
in a dense enough $r$-graph. This lower bound may be of independent interest.

\begin{lemma} \label{basic-supersaturation}
Let $r$ be an integer with $r\geq 2$. Let $H$ be an $r$-partite $r$-graph with $h$ vertices.
Let $\alpha$ and $A$ be positive reals satisfying that $A\geq r^{2r}$ and that $\ex(n,H)\leq An^\alpha$ for all $n$. There exists a constant $c_H>0$ such that the following holds.
Let $G$ be an $n$-vertex graph with $kn^\alpha$ edges where $k\geq 2^{3r} A$.
Then $G$ contains at least $c_H[e(G)]^{\frac{h-\al}{r-\al}} / n^{\frac{\al(h-r)}{r-\al}} = c_H k^{ \frac{h - \alpha}{r - \alpha}}n^{\alpha}$  copies of $H$.
\end{lemma}
\begin{proof} 
Let $p$ be a real  such that
\[  (8A/k)^{\frac{1}{r-\alpha}} \leq p\leq 2(8A/k)^{\frac{1}{r-\alpha}}
\quad \mbox{ and } \quad np\in \mathbb{Z}^+ .\]
Such $p$ exists since 
\[(8A/k)^{\frac{1}{r-\al} }\geq (8A/n^{r-\al})^{\frac{1}{r-\al}}\geq \frac{r^2}{n},\]
by our condition on $A$. Since $k\geq 2^{3r} A$, $p\in (0,1)$.
Let $W$ be a uniform random subset of $V(G)$ of size $w=np$.  
By our choice of $p$,  we have $w\geq r^2$.
By Lemma~\ref{binom-estimate},
\begin{equation} \label{eq1}
\E[e(G[W])]= e(G) \binom{n-r}{w-r}/\binom{n}{w}\geq \frac{1}{2} e(G) (w/n)^r=\frac{1}{2} e(G) p^r.
\end{equation}
Let $t=\binom{n}{w}$ and $W_1,W_2,\dots, W_t$ be all the $\binom{n}{w}$ subsets of $V(G)$ of size $w$. For each $i\in [t]$, let $G_i=G[W_i]$.
Let $\I_{good}$ be the set of $i\in [t]$ such that $e(G_i)\geq \frac{1}{4}mp^r$ and $\I_{bad}=[t]\setminus \I_{good}$.
Then $\sum_{i\in \I_{bad}} e(G_i)\leq \binom{n}{w} \frac{1}{4}e(G)p^r$ and hence by \eqref{eq1}
\begin{equation} \label{eq2}
\sum_{i\in \I_{good}} e(G_i) \geq \binom{n}{w}\frac{1}{4}e(G)p^r.
\end{equation}
For each $i\in \I_{good}$, we have
\[e(G_i)\geq \frac{1}{4} e(G) p^r=\frac{1}{4}kn^\alpha p^r=\frac{1}{4}kp^{r-\alpha} (np)^{\alpha}=\frac{1}{4}kp^{r-\alpha} w^\alpha
\geq 2Aw^\alpha,\]
where the last inequality holds since $p\geq (\frac{8A}{k})^{\frac{1}{r-\alpha}}$.
Hence $e(G_i)\geq 2Aw^\alpha>2\ex(w,H)$. By Lemma~\ref{simple-bound}, $G_i$ contains at least $e(G_i)-\ex(n,H)\geq 
\frac{1}{2}e(G_i)$ copies of $H$. 
Let $\lam$ denote the number of copies of $H$ in $G$.
Then, using \eqref{eq2}, we have
\[\lam\geq \frac{1}{\binom{n-h}{w-h}} \sum_{i\in \I_{good} } \frac{1}{2} e(G_i) \geq \frac{1}{8} \frac{\binom{n}{w}}{\binom{n-h}{w-h}} e(G) p^r
\geq \frac{1}{8} (n/w)^h e(G) p^r =\frac{1}{8} e(G) p^{r-h}= \frac{1}{8}e(G)(1/p)^{h-r}.
\]
Since $k=e(G)/n^\al$ and $p\leq 2(8A/k)^{\frac{1}{r-\al}}$, we have
\[\lam\geq \frac{1}{8} e(G) (1/2)^{h-r}(k/8A)^{\frac{h-r}{r-\alpha}}\geq c_H [e(G)]^{\frac{h-\al}{r-\al}} / n^{\frac{\al(h-r)}{r-\al}},\]
for some constant $c_H>0$.
\end{proof}

While Lemma~\ref{basic-supersaturation} gives a reasonably dense family of copies of $H$ in a dense enough host graph $G$, it is generally not as dense as what is conjectured in Conjecture~\ref{ES-conjecture-supersat}. Next, we present our key lemma, which is the basis of our main results in this paper.

\begin{lemma} [Key Lemma] \label{key-lemma}
Let $r,h,\ell$ be positive integers, where $h\geq r\geq 2$.
Let $H$ be an $r$-partite $r$-graph with $h$ vertices and $\ell$ edges. Let $\alpha, A$ be positive reals such that for each $n$ every 
$n$-vertex graph $G$ with $m\geq An^\alpha$ edges contains  at least $f(n,m)$ copies of $H$,
where $f$ is a function satisfying the following.
\begin{enumerate}
\item There is a constant $\delta>0$ such that for all $p\in (0,1]$ and  all positive reals $n,m$ 
\[ f(n,m)\geq \delta m^{\frac{h-\al}{r-\al}} / n^{\frac{\al(h-r)}{r-\al}} \quad \mbox{ and  }  \quad f(np,mp^r)\geq \delta f(n,m) p^h.\]
\item For fixed $n$, $f(n,m)$ is increasing and convex in $m$.
\end{enumerate}
Then there exist constants $k_0=k_0(H)$ and $C=C(H) \geq 1$ such that if $G$ is an $n$-vertex $r$-graph with $kn^\alpha$ edges, where $k\geq k_0$, then $G$ contains
a family $\F$ of copies of $H$ satisfying that
\begin{enumerate}
\item $|\F|\geq \delta f(n,\frac{1}{8}e(G))$.
\item $\forall S\subseteq E(G), 1\leq |S|\leq \ell -  1$, \[d_\F(S)\leq C k^{- \lambda^*(\alpha, H)(|S| - 1)} \frac{|\F|}{e(G)},\]
$\lam^*(\al, H)=\frac{1}{m^*_r(H)(r-\al)}$ and
$m^*_r(H)$ is the proper $r$-density of $H$.
\end{enumerate}
\end{lemma}
\begin{proof}
First we define some constants.
 Let $k_0$ be a sufficiently large constant, depending on $H$ and $A,\al$, such that the statement
after \eqref{p-definition} holds. 
%\[\beta=  (ck)^{-\frac{1}{\mu_0(H)(r-\alpha)}}, \quad  p=\eps \beta^{\mu_0(H)}.\]
Let $N=\delta f(n,\frac{1}{8} e(G))$. 
By our assumption about $f$ and that $n^\al=e(G)/k$,
\begin{equation} \label{N-lowerbound}
N\geq \delta \cdot \delta (\frac{1}{8} e(G))^{\frac{h-\al}{r-\al}}/n^{\frac{\al(h-r)}{r-\al}}
=\delta^2 (\frac{1}{8})^{\frac{h-\al}{r-\al}} k^{\frac{h-r}{r-\al}} e(G) =\delta' k^{\frac{h-r}{r-\al}} e(G) ,
\end{equation}
where $\delta'=\delta^2 (\frac{1}{8})^{\frac{h-\al}{r-\al}}$. 
Let \begin{equation} \label{constants-definition}
C= \max\{ 2^{\ell+h+2} r^{2h} (8A)^{\frac{h}{r-\al}}\ell, \frac{4}{\delta'} \} .
\end{equation}
Trivially, $C \geq 1$.
For convenience, let 
\begin{equation}\label{beta-definition}
\beta= k^{-\lam^*(\al,H)}
\end{equation}
where $\lam^*(\alpha, H)$ is as defined in condition 2. 
%with $\lam^*(\al, H)=\frac{1}{m^*_r(H)(r-\al)}$ and $m^*_r(H)$ is the proper $r$-density of $H$.
%By our choice of $\beta$, we have
%\begin{equation} \label{beta-lowerbound}
%\beta^{\ell-1} \cdot \frac{N}{e(G)}\geq  1.
%\end{equation}

To prove Lemma~\ref{key-lemma}, it suffices to find a family $\F$ with $|\F| = N$ such that the following condition holds \begin{equation} \label{codegree-condition}
\forall S\subseteq E(G), 1\leq |S|\leq \ell-1, d_\F(S)\leq C\beta^{|S|-1} \frac{N}{e(G)}.
\end{equation}

%This is because, by \eqref{beta-lowerbound},  the condition in \eqref{codegree-condition} also would extend to $|S| = \ell$. 
To build such a family, let $\F=\emptyset$.  We show that as long as $|\F|<N$, we can find a new copy of $H$ in $G$ to add to $\F$ so that $\eqref{codegree-condition}$ holds.

Clearly, initially $\F$ satisfies \eqref{codegree-condition}.
Given a set $S\subseteq E(G)$, where $1\leq |S|\leq \ell-1$,
we call $S$ {\it saturated} if  $d_\F(S)\geq C \beta^{|S|-1} \frac{N}{2e(G)}$. 
%Note that by \eqref{beta-lowerbound}, for all $S$ with $1\leq |S|\leq \ell-1$, we have

Recall that $m_r^*(H) = \max_{F \subsetneq H} \{ \frac{e(F) - 1}{v(F) - r}\} \geq  \frac{\ell - 2}{h - r}$.  Hence, we have that $\lam^*(\alpha, H) \leq \frac{(h - r)}{(\ell - 2)( r - \alpha)}$. Combining this with \eqref{N-lowerbound}, we have 

\begin{equation}\label{saturated-lowerbound}
C \beta ^{|S| - 1} \frac{N}{2e(G)} \geq C  k^{-\lam^*(\al, H)(\ell - 2) } \frac{\delta' k^{\frac{h - r}{ r - \al} }e(G)}{2 e (G)}  \geq \frac{C \delta '}{2} \geq 2.
%C\beta^{|S|-1}\frac{N}{2e(G)} \geq \frac{C}{\beta} \beta^{\ell-1} \frac{N}{2e(G)}\geq \frac{C}{2\beta}\geq 2.
\end{equation}

For each $i=1,\dots, \ell-1$, let $B_i$ denote the family of saturated $i$-subsets of $E(G)$.
We will call a copy of $H$ in $G$ a {\it good} copy of $H$ if it does not contain any member of $\bigcup_{i=1}^{\ell-1} B_i$.

\medskip

{\bf Claim 1.} For each $i=1,\dots, \ell-1$, $|B_i|\leq (2^\ell/C) (1/\beta)^{i-1} e(G)$.

\medskip

{\it Proof of Claim 1.} Let $\mu$ denote the number of pairs $(H',S)$ where $H'$ is a 
member of $\F$ and
$S\in B_i$ and $S\subseteq E(H')$. If we count $\mu$ by $S$, then by definition, 
\[\mu\geq |B_i| C \beta^{i-1} \frac{N}{2e(G)}.\]

On the other hand, if we count $\mu$ by $H'$, then 
\[\mu\leq |\F|\binom{\ell}{i}< 2^{\ell-1} N.\]
The claim follows by combining the last two inequalities and solving for $|B_i|$. 
\hfill $\Box$

Let $q=(8A)^{\frac{1}{r-\al}}r^2 \beta^{m^*_r(H)}$.
Since $\beta = k^{-\lam^*(\al, H)}$, we have
\begin{equation} \label{q-lowerbound}
q\geq (8A)^{\frac{1}{r-\al}} r^2 k^{-\lam^*(\al, H)m^*_r(H)}=  (8A)^{\frac{1}{r-\al}} 
r^2 k^{-\frac{1}{r-\al}}
\geq r^2(n^{r-\al})^{-\frac{1}{r-\al}} =\frac{r^2}{n}.
\end{equation}
Let $p$ be a positive real such that
\begin{equation} \label{p-definition}
q\leq p\leq 2q \quad 
\mbox{ and } \quad np\in \mathbb{Z}^+.
\end{equation}
Since $q\geq r^2/n$, it is easy to see such a $p$ exists. Because $k\geq k_0$,
by choosing $k_0$ to be large enough constant depending on $H$ and $A$, we
can ensure that $2q<1$ and hence $p\in (0,1)$.

Let $W$ be a uniform random subset of $V(G)$ of size $w=np$. By \eqref{q-lowerbound}, $w\geq r^2$.  By Lemma~\ref{binom-estimate},
\begin{equation} \label{GW-expectation}
\E[e(G[W])]=e(G)\frac{\binom{n-r}{np-r}}{\binom{n}{np}}\geq \frac{1}{2} e(G) p^r.
\end{equation}

For each $i=1,\dots, \ell-1$, let $Y_i(W)$ denote the set of members of $B_i$ that are contained in $W$.
Fix any $i=1,\dots, \ell-1$. Consider any member $S$ of $B_i$. Suppose $S$ spans $v_S$ vertices.
Then by the definition of $m^*_r(H)$, we have $\frac{i-1}{v_S-r}\leq m^*_r(H)$ and hence
$v_S\geq r+\frac{i-1}{m^*_r(H)}$. Hence,
\[\Pb[S\subseteq W]=\binom{n-v_S}{np-v_S}/\binom{n}{np}\leq p^{v_S} \leq p^{r+\frac{i-1}{m^*_r(H)}}.\]

This, along with Claim 1, implies that  for each $i=1,\dots, \ell-1$, 

\begin{equation} \label{Yi-expectation}
\E[|Y_i(W)|\leq |B_i| p^{r+\frac{i-1}{m^*_r(H)}} \leq (2^\ell/C) (1/\beta)^{i-1} p^{r+\frac{i-1}{m^*_r(H)}} e(G)
=(2^\ell/C) (1/\beta)^{i-1}p^{\frac{i-1}{m^*_r(H)}}\cdot e(G)p^r.
\end{equation}

Hence, by \eqref{p-definition} and \eqref{Yi-expectation}, for each $i=1,\dots,\ell-1$, we have

\[ \E[|Y_i(W)|]\leq (2^\ell/C) [2(8A)^{\frac{1}{r-\al}}r^2]^{\frac{i-1}{m^*_r(H)}} \cdot e(G)p^r <
(2^\ell/C) [2(8A)^{\frac{1}{r-\al}}r^2]^h \cdot e(G) p^r.
\]
By our choice of $C$, given in \eqref{constants-definition}, we have
\begin{equation} \label{YIW-bound}
\E[|\bigcup_{j=1}^{\ell-1}Y_j(W)|]\leq (\ell-1)   (2^\ell/C)[2(8A)^{\frac{1}{r-\al}}r^2]^h   \cdot e(G) p^r
\leq \frac{1}{4} e(G) p^r.
\end{equation}

By \eqref{GW-expectation} and \eqref{YIW-bound}, 
\begin{equation} \label{surplus-expectation}
\E[e(G[W]) -|\bigcup_{j=1}^{\ell-1}Y_j(W)|]\geq \frac{1}{2}e(G) p^r- \frac{1}{4} e(G) p^r
\geq \frac{1}{4} e(G) p^r,
\end{equation}
Let $t=\binom{n}{np}$ and $W_1,\dots, W_t$ be all the $np$-subsets of $V(G)$.
Let $\I_{good}$ be the set of $i\in [t]$ such that $e(G[W_i]) -|\bigcup_{j=1}^{\ell-1} Y_j(W_i)| 
\geq \frac{1}{8} e(G) p^r$. Let $\I_{bad}=[t]\setminus \I_{good}$.
Then 
\[\sum_{i\in \I_{bad}} [e(G[W_i]) -|\bigcup_{j=1}^{\ell-1} Y_j(W_i)|] \leq \binom{n}{np} \frac{1}{8}e(G) p^r.\]
Hence, by \eqref{surplus-expectation},
\begin{equation} \label{good-surplus}
\sum_{i\in \I_{good}}  e(G[W_i]) -|\bigcup_{j=1}^{\ell-1} Y_j(W_i)| \geq \binom{n}{np} \frac{1}{8}e(G) p^r.
\end{equation}

For each $i\in \I_{good}$, let $G'_i$ be a subgraph of $G[W_i]$ obtained by deleting an edge from 
each member of $\bigcup_{j=1}^{\ell-1} Y_j[W_i]$.  By \eqref{good-surplus}, 
\begin{equation} \label{good-sum}
\sum_{i\in \I_{good}}  e(G'_i)\geq \binom{n}{np} \frac{1}{8}e(G) p^r.
\end{equation}
By the definition of $\I_{good}$, for each $i\in \I_{good}$,
we have
\begin{equation} \label{individual-good}
e(G'_i)\geq \frac{1}{8}e(G) p^r =\frac{1}{8} kn^\alpha p^r =\frac{1}{8} k p^{r-\alpha} (np)^\alpha
=\frac{1}{8} k p^{r-\alpha} w^\alpha.
\end{equation}

By \eqref{q-lowerbound} and \eqref{p-definition},
\[p\geq q\geq (8A)^{\frac{1}{r-\al}} 
r^2 k^{-\frac{1}{r-\al}}.\]

Hence, by \eqref{individual-good}, for each $i\in \I_{good}$,
\[e(G'_i)\geq \frac{1}{8} k (8A)r^{2(r-\al)}k^{-1}w^\al
\geq Aw^\alpha.\]

Hence, by our assumption about $H$, for each $i\in \I_{good}$, $G'_i$ contains at least $f(w,e(G'_i))=f(np, e(G'_i))$ copies of
$H$ in $G$. Now, the crucial observation is that any copy $H'$ of $H$ in $G'_i$, where $i\in \I_{good}$, is a good copy of $H$ in $G$.
Indeed, suppose $H'$ contains a member $S$ of $B_j$ for some $j=1,\dots, \ell-1$, then $S\subseteq E(G'_i)\subseteq E(G[W_i])$.
So $S\in Y_j[W_i]$. But in forming $G'_i$ from $G[W_i]$ we have removed an edge from each member of $Y_j[W_i]$ and hence $S\not\subseteq E(G'_i)$, a contradiction.

Let $\cH_{good}$ denote the family of good copies of $H$ in $G$. By our discussions above, \eqref{good-sum} and 
our assumptions about the function $f$, we have
\begin{eqnarray*}
|\cH_{good}|&\geq&  \frac{1} {\binom{n-h}{np-h}}\sum_{i\in \I_{good}} f(np, e(G'_i))\geq\frac{1} {\binom{n-h}{np-h}}\binom{n}{np} f(np, \frac{1}{8}e(G)p^r)\\
&\geq& (1/p)^h \delta f(n,\frac{1}{8}e(G)) p^h=\delta f(n,\frac{1}{8}e(G)).
\end{eqnarray*}
Hence, $|\cH_{good}|\geq N>|\F|$. So, there must exist a member $H'$ of $\cH_{good}$ that is not in $\F$. Let us add $H'$ to $\F$.
Consider any subset $S\subseteq E(G)$, where $1\leq |S|\leq \ell-1$. If $S$ is not contained in $H'$ then $d_\F(S)$ is unchanged.
If $S\subseteq E(H')$, then since $H'$ is good, $S$ is unsaturated prior to the addition of $H'$ and hence now satisfies
$d_\F(S)\leq C\beta^{|S|-1}\frac{N}{2e(G)}+1\leq C\beta^{|S|-1}\frac{N}{e(G)}$, by \eqref{saturated-lowerbound}. Hence \eqref{codegree-condition} still holds for the new family $\F$.
Thus, we can iterate the process to find a family $\F$ that satisfies \eqref{codegree-condition} and such that $|\F|\geq N$.
\end{proof}

For convenience, we will refer to the $\beta$ associated with the family $\F$ as {\it the codegree ratio} of $\F$.

By Lemma~\ref{basic-supersaturation} and Lemma~\ref{key-lemma}, we obtain the following general balanced supersaturation theorem, which implies our first main theorem, 
Theorem~\ref{main1}.
 
\begin{theorem} \label{balanced-supersat1}
Let $r$ be an integer with $r\geq 2$. Let $H$ be
an $r$-partite $r$-graph with $h$ vertices and $\ell$ edges.
Let $\alpha$ and $A$ be positive reals satisfying that $A\geq r^{2r}$, $\alpha>r-\frac{1}{m_r(H)}$ and that $\ex(n,H)\leq An^\alpha$ for all $n$,
There exist constants $\delta_H,k_0=k_0(H), C=C(H)>0$ such that the following holds.
Let $G$ be an $n$-vertex graph with $kn^\alpha$ edges where $k\geq k_0$.
Then $G$ contains a family $\F$ of copies of $H$ satisfying that
\begin{enumerate}
\item $|\F|\geq \delta_H[e(G)]^{\frac{h-\al}{r-\al}} / n^{\frac{\al(h-r)}{r-\al}}$,
\item  $\forall S\subseteq E(G), 1\leq |S|\leq \ell$ \[d_\F(S)\leq C k^{-\lam(\al, H)(|S|-1)}\frac{|\F|}{e(G)},\]
where 
$\lam(\al, H)=\frac{1}{m_r(H)(r-\al)}$ and $m_r(H)$ is the $r$-density of $H$.
\end{enumerate}
\end{theorem}
\begin{proof}
By choosing $k_0$ to be at least $2^{3r}$, by Lemma~\ref{basic-supersaturation}, $H$ has the
property that every $n$-vertex graph with $kn^\al$ edges, where $k\geq k_0$ contains
at least  $c_H[e(G)]^{\frac{h-\al}{r-\al}} / n^{\frac{\al(h-r)}{r-\al}}$ copies of $H$.
Let $f(n,m)= c_Hm^{\frac{h-\al}{r-\al}} / n^{\frac{\al(h-r)}{r-\al}}$.
It is straightforward to  see that $ f(np,mp^r)=f(n,m) p^h$, for any $p$.
Furthermore, for fixed $n$, $f(n,m)$ is clearly increasing and convex in $m$.
By Lemma~\ref{key-lemma},  there exist constants $k_0$ and $C'$ such that for every $n$ if $G$ is an $n$-vertex $r$-graph with $kn^\alpha$ edges, where $k\geq k_0$, then $G$ contain a family $\F$ of copies of $H$ satisfying
\begin{enumerate}
\item [{\rm (a)} ] $|\F|\geq c_H f(n,\frac{1}{8}e(G))$.
\item [{\rm (b) }] $\forall S\subseteq E(G), 1\leq |S|\leq \ell $, \[d_\F(S)\leq C' \beta^{|S|-1}\frac{|\F|}{e(G)},\]
where 
\[\beta=\max\left\{ \left(\frac{e(G)}{|\F|}\right)^{\frac{1}{\ell-1}}, k^{-\lam(\al, H)}\right\},\]
$\lam(\al, H)=\frac{1}{m_r(H)(r-\al)}$. Note that here we used the fact that $\lambda(\alpha, H) \leq \lambda^*(\alpha, H)$, as to make the conclusion work for sets of size $\ell$, we switch to $\lambda$.
\end{enumerate}

Since $e(G)=kn^\al$, by condition (a) above, 
\begin{equation} \label{F-lowerbound}
|\F|\geq c_H [e(G)]^{\frac{h-\al}{r-\al}} / n^{\frac{\al(h-r)}{r-\al}}= 
c_H k^{\frac{h-r}{r-\al}} e(G).
\end{equation}
For convenience, we may further assume that $c_H<1$.
By definition of $r$-density, $\frac{\ell-1}{h-r}\leq m_r(H)$. Hence, by \eqref{F-lowerbound},
\begin{equation} \label{MF-upperbound}
\left(\frac{e(G)}{|\F|}\right)^{\frac{1}{\ell-1}} \leq (c_H^{-1} k^{-\frac{h-r}{r-\al}})^{\frac{1}{\ell-1}} \leq (c_H)^{-\frac{1}{\ell-1}}
k^{-\lam(\al,H)}.
\end{equation}

By the definition of $\beta$ in condition (b) and \eqref{MF-upperbound}, 

\[\beta\leq c_H^{-\frac{1}{\ell-1}}  k^{-\lam(\al, H)}.\]

Let $C=C' (c_H)^{-1}$. By condition (b) above, we have
$\forall S\subseteq E(G), 1\leq |S|\leq \ell$, 
\[d_\F(S)\leq C' \beta^{|S|-1}\frac{|\F|}{e(G)}
\leq C' (c_H^{-\frac{|S|-1}{\ell-1}}  k^{-\lam(\al, H)( |S|-1)|}) \frac{|\F|}{e(G)}
\leq C k^{-\lam(\al, H)(|S|-1)|} \frac{|\F|}{e(G)}.\]
So, the theorem holds with $\delta_H=c_H$.
\end{proof}

If one applies Theorem~\ref{balanced-supersat1} to an adaption of Proposition~\ref{MS-proposition}, one can retrieve
one of the enumeration results of Ferber, McKinley and Samotij \cite{FMS} as below.  The difference is that Ferber, McKinley and Samotij
used a weaker version of balanced supersaturation.

\begin{corollary}
Let $H$ be an $r$-uniform hypergraph and let $\al$ and $A$ be positive constants. Suppose that 
$\al>r-\frac{1}{m_r(H)}$ and that $\ex(n,H)\leq An^\al$ for all $n$. Then there exists a constant $C$ depending 
only on $\al, A$, and $H$ such that for all $n$, there are at most $2^{Cn^\al}$ $H$-free $r$-uniform hypergraphs on $n$ vertices.
\end{corollary}

Next, we  give a stronger balanced supersaturation theorem for Erd\H{o}s-Simonovits good bipartite graphs. 
This implies our second main result, Theorem~\ref{main2}.

\begin{theorem} \label{balanced-supersat2}
Let $H$ be a bipartite graph with $h$ vertices and $\ell$ edges.
Suppose that $H$ is Erd\H{o}s-Simonovits good.
Let $\alpha$ and $A$ be positive reals satisfying that $A\geq 16$, $\alpha>2-\frac{1}{m_2(H)}$ and that $\ex(n,H)\leq An^\alpha$ for all $n$. There exist constants $\delta_H,k_0=k_0(H), C=C(H)>0$ such that the following holds.
Let $G$ be an $n$-vertex graph with $kn^\alpha$ edges where $k\geq k_0$.
Then $G$ contains a family $\F$ of copies of $H$ satisfying that
\begin{enumerate}
\item $|\F|\geq \delta_H [e(G)]^{\ell}/n^{2\ell-h}$,
\item  $\forall S\subseteq E(G), 1\leq |S|\leq \ell$, \[d_\F(S)\leq C \beta^{|S|-1}\frac{|\F|}{e(G)},\]
where 
\[\beta= \max\{ k^{-1} n^{-\phi(\al, H)}, k^{-\lam^*(\al, H)}\},\]
$\phi(\al, H)=\frac{\al\ell-\al+h-2\ell}{\ell-1}$,
$\lam^*(\al,H)=\frac{1}{m^*_2(H)(2-\al)}$ and $m^*_2(H)$ is the proper $2$-density of $H$.
\end{enumerate}
\end{theorem}
\begin{proof}
Since $H$ is Erd\H{o}s-Simonovits good, there exist constants $c_H, k_1>0$ such that
 $n$-vertex graph with $m=kn^\al$ edges, where $k\geq k_1$ contains
at least  $c_Hm^\ell/n^{2\ell-h}$ copies of $H$.
Let $f(n,m)= c_Hm^\ell / n^{2\ell-h}$. Using the facts $\al\geq 2-\frac{1}{m_r}\geq 2-\frac{h-2}{\ell-1}$
and $k\leq 2-\al$, one can show that
$f(n,m)\geq c_H m^{\frac{h-\al}{2-\al}}/n^{\frac{\al(h-2)}{2-\al}}$.
Also, it is straightforward to  see that $ f(np,mp^2)=f(n,m) p^h$, for any $p$.
Furthermore, for fixed $n$, $f(n,m)$ is clearly increasing and convex in $m$.
By Lemma~\ref{key-lemma},  there exist constants $k_0$ and $C'$ such that if $G$ is an $n$-vertex $r$-graph with $kn^\alpha$ edges, where $k\geq k_0$, then $G$ contain a family $\F$ of copies of $H$ satisfying
\begin{enumerate}
\item [{\rm (a)} ] $|\F|\geq c_H f(n,\frac{1}{8}e(G))$.
\item [{\rm (b) }] $\forall S\subseteq E(G), 1\leq |S|\leq \ell$, \[d_\F(S)\leq \hc\hb^{|S|-1}\frac{|\F|}{e(G)},\]
where 
\[\hb=\max\left\{ \left(\frac{e(G)}{|\F|}\right)^{\frac{1}{\ell-1}}, k^{-\lam^*(\al, H)}\right\},\]
$\lam^*(\al,H)=\frac{1}{m^*_2(H)(2-\al)}$ and
$m^*_2(H)$ is the proper $2$-density of $H$.
\end{enumerate}

By condition (a) above, 
\begin{equation} \label{F-lowerbound2}
|\F|\geq c_H [e(G)]^{\ell}/n^{2\ell-h}.
\end{equation} 
For convenience, we may further assume that $c_H<1$.
Hence,
\[|\F|/e(G)\geq c_H [e(G)]^{\ell-1}/n^{2\ell-h}=c_H k^{\ell-1}n^{\al\ell-\al+h-2\ell}
=c_H k^{\ell-1}n^{\phi(\al, H)(\ell-1)}.\]
Hence, by the definition of $\hb$ given in condition (b) above,
\[\hb\leq \max \{c_H^{-\frac{1}{\ell-1}}k^{-1} n^{-\phi(\al, H)}, k^{-\lam^*(\al, H)}\}.\]
Let 
\[\beta=\max\{ k^{-1} n^{-\phi(\al, H)}, k^{-\lam^*(\al, H)}\}
\quad \mbox{ and } \quad C=\hc\cdot c_H^{-1}.\]
It is straightforward to verify that condition 2 holds for these choices of $\beta$ and $C$.
So the theorem holds with $\delta_H=c_H$.
\end{proof}

The advantage of Theorem~\ref{balanced-supersat2} over Theorem~\ref{balanced-supersat1} is
that for Erd\H{o}s-Simonovits good $H$, the former produces a $\beta$ value that is no larger than the
former and in many cases produces a smaller $\beta$, at least for a suitable range of $k$.
Next, we show that the method developed in this section allows us to give a short proof of Theorem~\ref{MS-cycle-supersat} 
of Morris and Saxton \cite{MS}.

\begin{theorem} [Restatement of Theorem~\ref{MS-cycle-supersat}] \label{c2ell-supersat}
For every $\ell\geq 2$, there exist constants $C>0, \delta>0$ and $k_0\in \mathbb{N}$
such that for each $n\in \mathbb{N}$ if $G$ is an $n$-vertex graph with $kn^{1+1/\ell}$ edges,
where $k_0\leq k$, then there exists a collection $\F$ of copies of $C_{2\ell}$ in $G$ such that
\begin{enumerate}
\item[{\rm (a) }] $|\F|\geq \delta k^{2\ell} n^2$,
\item[{\rm (b) }] $\forall S\subseteq E(G), 1\leq |S|\leq 2\ell - 1$, $d_{\F}(S)\leq Ck^{2\ell-|S|-\frac{|S|-1}{\ell-1}}
n^{1-1/\ell}$.
\end{enumerate}
\end{theorem}
\begin{proof}
It is well known that $C_{2\ell}$ is Erd\H{o}s-Simonovits good for $\al=1+1/\ell$ and some $A>0$, so we can apply Lemma~\ref{key-lemma} with $f(n, m) = c \left(\frac{m}{n}\right)^{2\ell}$ with $c$ some positive constant depending only on $C_{2 \ell}$. Let $G$ be an $n$-vertex graph with $kn^{\al}$ edges, with 
$k_0\leq k$. Let $\F$ be the family of copies of $H$ in $G$
guaranteed by Lemma~\ref{key-lemma}.
 One can check that $m^*_2(C_{2\ell})=1$
and hence $\lam^*(\al, C_{2\ell})=\frac{1}{1\cdot (2-\al)}=\frac{\ell}{\ell-1}$. 
Condition (a), (b) readily follow from conditions (1),(2) of $\F$
guaranteed in Lemma~\ref{key-lemma}.
\end{proof}

%%%%%%%%%%%%%%%%%%%%%%%%%%%%%%%%%%%%%%%%%%%%%%%%%%%%%%%

%%%%%%%%%%%%%%%%%%%%%%%%%%%%%%%%%%%%%%%%%%%%%%%%%%%%%%%

\section{Applications to the Tur\'an problem in random graphs} 
\label{sec:random-Turan}

In this section, we apply our balanced supersaturation results to obtain some general
bounds on the Tur\'an number of a bipartite graph $H$ in the Erd\H{o}s-R\'enyi random graph $G(n,p)$.
In fact, once we have Theorem~\ref{balanced-supersat1}
and Theorem~\ref{balanced-supersat2},  the corresponding random Tur\'an results will readily follow
using the framework set up by Morris and Saxton \cite{MS}. Nevertheless, for completeness, we will include all the
technical details. The framework is based on the container method pioneered by Balogh, Morris, and Samotij \cite{BMS}
and independently by Saxton and Thomason \cite{ST}.

\begin{definition} \label{codegree-definition}
Given an $r$-graph $\F$, define the {\it co-degree  function} of $\F$
\[\delta(\F,\tau)=\frac{1}{|\F|}\sum_{j=2}^r \frac{1}{\tau^{j-1}}\sum_{v\in V(\F)} d^{(j)}(v),\]
where 
\[d^{(j)}(v)=\max\{ d_{\F} (S): v\in S\subseteq v(\F) \mbox{ and } |S|=j\}.\]
\end{definition}

We need the following theorem from 
Morris and Saxton \cite{MS}, which is a quick consequence of analogous theorems of Balogh, Morris and Samotij
\cite{ST} Theorem 6.2 and of Saxton and Thomason \cite{BMS} Proposition 3.1.

\begin{theorem} \label{container-lemma}
Let $r\geq 2$ be an integer. Let $0<\delta<\delta_0(r)$ be a sufficiently small real. Let $\F$ be an $r$-graph
with $N$ vertices. Suppose that $\delta(\F, \tau)\leq \delta$ for some $\tau>0$. Then there exists a collection
$\cC$ of subsets of $V(\F)$ and a function $f:V(\F)^{(\leq \tau N/\delta)}  \to \cC$ such that
\begin{enumerate} 
\item [{\rm (a)}] For every independent set $I$ in $\F$ there exists $T\subset I$ with $|T|\leq \tau N/\delta$
and $I\subset f(T)$. 
\item [{\rm (b)}] $e(\F[C])\leq (1-\delta) e(\F)$ for every $C\in \cC$.
\end{enumerate}
\end{theorem}

\begin{proposition} \label{container-prop}
Let $H$ be a bipartite graph with $h$ vertices and $\ell$ edges. 
Let $\alpha, A$ be positive reals satisfying that for each $n\in\mathbb{N}$  every 
$n$-vertex graph $G$ with $m\geq An^\alpha$ edges contains  at least $f(n,m)$ copies of $H$,
where $f$ is a function satisfying the following.
\begin{enumerate}
\item There is a constant $\delta>0$ such that for all $p\in (0,1]$ and  all positive reals $n,m$ 
\[ f(n,m)\geq \delta m^{\frac{h-\al}{2-\al}} / n^{\frac{\al(h-2)}{2-\al}} \quad \mbox{ and  }  \quad f(np,mp^2)\geq \delta  f(n,m) p^h.\]
\item For fixed $n$, $f(n,m)$ is increasing and convex in $m$.
\end{enumerate}
Let $k_0, C$, $\F$ and $\beta$ be as guaranteed by Lemma~\ref{key-lemma}.
 There exist $k^*_0\in \mathbb{N}$ 
and a real $\eps>0$ such that the following holds for every $k\geq k^*_0$ and every $n\in \mathbb{N}$.
Set
\[\mu=\frac{k\beta}{\eps}. \] 
Given a graph $G$ with $n$ vertices and $kn^{\al}$ edges, there exists a 
function $f_G$ that maps subgraphs of $G$ to subgraphs of $G$ such that for every
$H$-free subgraph $I\subset G$,
\begin{enumerate}
\item[{\rm (a)} ]There exists a subgraph $T=T(I)\subset I$ with $e(T)\leq \mu n^\al$ and $I\subset f_G(T)$, and
\item[{\rm (b)} ] $e(f_G(T(I)))\leq (1-\eps) e(G)$.
\end{enumerate}

\end{proposition}
\begin{proof}
Note that condition 1 in the proposition still holds if we replace $\delta$ with an even smaller positive real.
Hence in the our proof, we may assume $\delta$ to be sufficiently small.
 Let $\F$ be the family guaranteed by  Lemma~\ref{key-lemma} with codegree ratio
$\beta$. Let $N=v(\F)=e(G)=kn^\al$. Since we will view $\F$ as a hypergraph, we will write $e(\F)$ for $|\F|$.
Set 
\[\frac{1}{\tau}=\frac{\delta^2} {\beta} \quad \mbox{ and } \quad \eps=\delta^3. \]

Since $\forall S\subseteq E(G), 1\leq |S|\leq \ell-1, d_\F(S)\leq C\beta^{|S|-1}e(\F)/N$ holds, we have
\begin{equation} \label{tau-bound1}
\frac{1}{e(\F)} \sum_{j=2}^{\ell-1} \frac{1}{\tau^{j-1}} \sum_{v\in V(\F)} d^{(j)} (v)
\leq \frac{1}{e(\F)} \sum_{j=2}^{\ell-1} (\frac{\delta^2}{\beta})^{j-1} \cdot  N\left(C\beta^{j-1} \frac{e(\F)}{N}\right)
<C\sum_{j=2}^{\ell-1}\delta^{2j-2}\leq 2C\delta^2.
\end{equation}

Also, since $\forall v\in V(\F)$, $d^{(\ell)}(v)\leq 1$ and $C\beta^{\ell-1}e(\F)/N\geq 1$, we have
\[\frac{1}{e(\F)}\cdot \frac{1}{\tau^{\ell-1}}\sum_{v\in V(F)} d^{(\ell)} (v)\leq \frac{N}{e(\F)} \cdot (\frac{\delta^2}{\beta})^{\ell-1} 
\leq C\delta^{2\ell-2}\leq C\delta^2,\]
By our discussion above, we get
\begin{equation} \label{delta-bound}
\delta(\F,\tau)=\frac{1}{e(\F)} \sum_{j=2}^{\ell} \frac{1}{\tau^{j-1}} \sum_{v\in V(\F)} d^{(j))} (v)
\leq 2C\delta^2+C\delta^2 \leq \delta.
\end{equation}
By Theorem~\ref{container-lemma}, there exist
a collection $\cC$ of subsets of $V(\F)$ and  a function $f_G: V(\F)^{(\leq \tau N/\delta)}\to \cC$
such that for every $H$-free subgraph $I\subset G$,

\begin{enumerate} 
\item [{\rm (a')}] there exists $T=T(I)\subset I$ with $e(T)\leq \tau N/\delta$
and $I\subset f_G(T)$, and
\item [{\rm (b')}] $e(\F[T(I)))\leq (1-\delta) e(\F)$.
\end{enumerate}

In condition (a') we have 
\[e(T)\leq \tau N/\delta= (\beta/\delta^5) e(G) = (\beta/\eps) kn^\al=\mu n^\al,\]
condition (a') is equivalent to condition (a).
To complete the proof, it suffices to show that if $I$ is an independent set in $\F$
(i.e. if $I$ is an $H$-free subgraph of $G$) we have
$e(f_G(T(I)))\leq (1-\eps) e(G)$. 
Let $D=f_G(T(I))$.
By condition $(b')$, $e(\F[D])\leq (1-\delta) e(\F)$.
Hence, if we delete $v(\F)\setminus D$ from $\F$ we lose at least $\delta e(\F)$ edges of $\F$. On the other hand,
by our assumption about $\F$, each vertex of $\F$ lies in at most $Ce(\F)/e(G)$ edges of $\F$. Hence if we delete
$V(\F)\setminus C$ from $\F$, we lose at most $(v(\F)-|D|) Ce(\F)/e(G)$ edges of $\F$. Hence, we have
\[\delta e(\F)\leq (v(\F)-|D|) C e(\F)/e(G).\]
Solving for $|D|$ and using $v(\F)=e(G)$, we have 
\[|D|\geq v(\F)-(\delta/C) e(G) =[1-(\delta/C)] e(G) \geq (1-\eps) e(G).\]
In other words, we have $e(f_G(T(I))\geq (1-\eps)e(G)$, as desired.
\end{proof}

\begin{remark} \label{container-remark}
When we apply Proposition~\ref{container-prop}, we can take $\beta=k^{-\lam(\al,H)}$ as in Theorem~\ref{balanced-supersat1}
and for Erd\H{o}s-Simonovits good $H$, we can take $\beta=\max\{k^{-1}n^{-\phi(\al,H)}, k^{-\lam^*(\al,H)}\}$ as
in Theorem~\ref{balanced-supersat2}.
\end{remark}

We need an estimation lemma from \cite{MS}.
\begin{lemma} [\cite{MS}] \label{sum-estimate}
Let $M>0, s>0$ and $0<\delta<1$. If $a_1,\dots, a_m$ are reals that satisfy $s=\sum_j a_j$ and $1\leq a_j\leq (1-\delta)^j M$
for each $j\in [m]$, then 
\[s\log s \leq \sum_{j=1}^m a_j\log a_j + O(M).\]
\end{lemma} 

In what follows, by a {\it colored graph}, we mean a graph together with a labelled partition of its edge set.
Next, we prove two similar theorems, by adapting the arguments given in Section 6 of Morris-Saxton \cite{MS}
to fit our general balanced supersaturation results.
The former applies to all bipartite graphs that contain a cycle and the latter applies to Erd\H{o}s-Simonovits
good bipartite graphs that contain a cycle.

\begin{theorem} \label{container-main}
Let $H$ be a bipartite graph with $h$ vertices and $\ell$ edges that contains a cycle. Let $A,\alpha$ be positive reals such that
$\ex(n,H)\leq An^\al$ holds for every $n\in \mathbb{N}$.  
There exists a constant $C$ such that the following holds for all sufficiently large $n\in\mathbb{N}$
 and $k\in \mathbb{R}^+$. Let $\I(n)$ denote all the $H$-free graphs on $[n]$
and $\G(n,k)$ the collection of all graphs on $[n]$ with at most $kn^\al$ edges.
There exists a collection $\cS$ of colored graphs
with $n$ vertices and at most $Ck^{1-\lam(\al,H)} n^\al$ edges and functions
\[g: \I\to \cS\quad \mbox { and } \quad h:\cS\to \G(n,k)\]
with the following properties
\begin{enumerate}
\item[{\rm (a)}] $\forall s\geq 1$ the number of colored graphs in $\cS$ with $s$ edges is at most
\[ \left(\frac{Cn^\al}{s}\right)^{\frac{\lam(\al, H)}{\lam(\al, H)-1}\cdot s}\cdot \expo\left( Ck^{1-\lam(\al, H)} n^\al\right) .\]
\item[{\rm (b) }] $\forall I\in \I(n)$ $g(I)\subset I \subset h(g(I))\cup g(I)$. 
\end{enumerate}
\end{theorem}

\begin{proof}
Note that since $H$ contains a cycle, $m_2(H)\geq 1$.
Let $I\in \I(n)$. We will apply Proposition~\ref{container-prop} repeatedly (with $\beta=k^{-\lam(\al,H)}, \mu=k\beta/\eps
=(1/\eps)k^{1-\lam(\al,H)}$).
 Let $G_0=K_n$. For sufficiently large $n$, $G_0$ clearly satisfies the condition on $G$ in Proposition~\ref{container-prop}. Apply Proposition~\ref{container-prop}, with $G_0$ playing the role of $G$
to obtain the function $f_{G_0}$ and a subset $T_1$ of $I$ with $T_1\subset I \subset f_{G_0}(T_1)$ where
$T_1$ and $f_{G_0}(T_1)$ satisfy the additional properties described in Proposition~\ref{container-prop}.
Now, let $G_1=f_{G_0}(T_1)\setminus T_1$ and $I_1=I\cap G_1=I\setminus T_1$.
Apply Proposition~\ref{container-prop} again, with $G_1$ playing the role of $G$ and $I_1$ playing the role of $I$
to obtain the function $f_{G_1}$ and a subset $T_2$ of $I_1$ with $T_2\subset I_1\subset f_{G_1}(T_2)$.
We continue like this until we arrive at a graph $G_{m(I)}$ with at most $kn^\al$ edges. 

Let $g(I) =  T_1 \cup T_2 \cup \cdots \cup T_{m(I)}$, where elements of $T_i$ are colored with color $i$.   Let $\cS(s) = \{g(I) :|g(I)| = s\}$.  Let $h(g(I)) = G_{m(I)}$. Note that $h$ is well-defined (see \cite{BMS, MS, ST} for detailed discussion).
 Furthermore,  as $g(I) \subset I \subset h(g(I)) \cup g(I)$,  conditions (b) is fulfilled.   
It remains to show (a).  We begin by partitioning $\cS(s)$ into sets $\cS_m(s)$ where $\cS_m(s) = \{ S \in \cS(s) : \text{ the edges of } S  \text{ are colored with } m \text{ colors}\}$.

For each $m\in \mathbb{N}$,  let

\[\K(m) = \{\bk = (k_1,  \cdots k_m): k_j \in \mathbb{R} ,  (1-\eps)^{j - m}k \leq k_j \leq (1 - \eps)^jn^{2 - \alpha} \text{ and } k_jn^{\alpha} \in \mathbb{N} \}\]

And for each $\bk \in \K(m)$,  let 

\[\A(\bk) = \{\ba= (a_1,  \dots a_m) : a_j \in \mathbb{N},  a_j \leq \frac{1}{\eps}k_j^{1 - \lambda(\al,H)} \text{ and } \sum_j a_j = s \}\]

By definition each sequence in $\bk \in \K(m)$ corresponds to a potential sequence $(G_1, \cdots G_m)$ where $e(G_j) = k_jn^{\alpha}$,  as the edges of $(1-\eps)^{j - m}kn^{\alpha} \leq e(G_j) \leq (1 - \eps)^j n^2$ and by Proposition~\ref{container-prop}, we have $ e(T_{j})\leq \frac{1}{\eps}k_j^{1 - \lam(\al, H)}n^{\alpha}$.  Note further that our algorithm returns pairs $(G_i, T_i)$,  such that each sequence of $(T_1,  \cdots T_m)$ is uniquely identified with a sequence $(G_1,  \cdots G_m)$.  Thus it suffices to only count the choices of $T_i$.   The sequence $\ba \in \A(\bk)$  corresponds to a sequence of sizes for $T_i$.  Thus,  we have that for a fixed $m$

\[|\cS_m(s)| \leq \sum_{\bk\in \K(m)} \sum_{\ba\in \A(\bk)} \prod_{j=1}^m \binom{k_j n^\al}{a_j}.\]

Since for each $j\in [m]$ 
\[a_j\leq (1/\eps) k_j^{1-\lam(\al,H)}  \leq  (1/\eps)[(1-\eps)^{-j+1} k]^{1-\lam(\al,H)} 
=(1/\eps) (1-\eps)^{(\lam(\al,H)-1)(j-1)} k^{1-\lam(\al,H)}.\]
So
\[k_j\leq (1/\eps)^{\frac{1}{\lam(\al, H)-1} } n/ a_j^{\frac{1}{\lam(\al, H)-1}}.\]
Hence

\[\binom{k_jn^\al}{a_j}\leq \left(\frac{ek_jn^\al}{a_j}\right)^{a_j}  \leq \left( \frac{n^\al}{\eps a_j } \right )^{\frac{\lam(\al,H)}{\lam(\al,H)-1}\cdot a_j}\]
Applying Lemma~\ref{sum-estimate} with $M=(1/\eps) k^{1-\lam(\al,H)} n^\al,1-\delta=(1-\eps)^{\lam(\al,H)-1}$, the product over $j=1,\dots, m$ is at most
\[\left ( \frac{C'n^\al}{s}\right) ^{\frac{\lam(\al, H)}{\lam(\al, H)-1}\cdot s}\cdot \expo\left ( C'k^{1-\lam(\al, H)} n^\al\right),\]
for some $C'=C'(H)$. 

Thus,   \[|\cS(s)| \leq \sum_{m = 1}^{\infty}\sum_{\bk\in \K(m)} \sum_{\ba\in \A(\bk)}\left( \frac{C'n^\al}{s}\right) ^{\frac{\lam(\al, H)}{\lam(\al, H)-1}\cdot s}\cdot \expo\left( C'k^{1-\lam(\al, H)} n^\al\right)\]

Note that $|\K(m)| = 0$ for values of $m \gg \log n $ as $(1- \eps)^m n^{2} < k$. 
For any fixed $m  = O(\log n)$,   $|\K(m)|  = n^{O(\log n)}$,  as when picking any $\bk$, we have $O(n^2)$ choices for each coordinate and  $O(\log n) $ coordinates to choose for.  For any fixed $\bk$,   similarly, $|\A(\bk)| = n^{O(\log n) }$.  
Hence, 
\[\sum_{m=1}^\infty \sum_{\bk\in \K(m)} |\A(\bk)|=n^{O(\log n)}=\expo(O(\log^2 n)).\] 
Since $m_2(H)\geq 1$, $\lam(\al,H)\leq 1/(2-\al)$. This imples $(2-\al)(1-\lam)\geq 1-\al$.
Since $k\leq n^{2-\al}$, we have $k^{1-\lam(\al,H)} n^\al\geq n\gg \log^2 n$.
Hence, the theorem holds by choosing an appropriate $C>C'$.
\end{proof}

\begin{remark}
The proof of Theorem~\ref{container-main} can be adapted for $r$-partite $r$-graphs as well. However, as the random Tur\'an problem
in hypergraphs is not a focus of this paper, we opt to state Theorem~\ref{container-main} just for bipartite graphs.
\end{remark}

\begin{theorem} \label{container-ES}
Let $H$ be a bipartite graph with $h$ vertices and $\ell$ edges that contain a cycle. Let $A,\alpha$ be positive reals such that
$\ex(n,H)\leq An^\al$ holds for every $n\in \mathbb{N}$. Suppose that $H$ is Erd\H{o}s-Simonovits good.
There exists a constant $C$ such that the following holds
for all sufficiently large $n\in\mathbb{N}$ and $k\in \mathbb{R}^+$ with 
\[k\leq n^{\frac{\phi(\al,H)}{\lam^*(\al, H)-1}} / (\log n)^{\frac{2}{\lam^*(\al, H)-1}}.\] Let $\I(n)$ denote all the $H$-free graphs on $[n]$
and $\G(n,k)$ the collection of all graphs on $[n]$ with at most $kn^\al$ edges.
There exists a collection $\cS$ of colored graphs
with $n$ vertices and at most $Ck^{1-\lam^*(H)} n^\al$ edges and functions
\[g: \I\to \cS\quad \mbox { and } \quad h:\cS\to \G(n,k)\]
with the following properties
\begin{enumerate}
\item[{\rm (a)}] $\forall s\geq 1$ the number of colored graphs in $\cS$ with $s$ edges is at most
\[ \left(\frac{Cn^\al}{s}\right)^{{\frac{\lam^*(\al, H)}{\lam^*(\al, H)-1}\cdot s}}\cdot \expo\left( Ck^{1-\lam^*(\al, H)} n^\al\right) .\]
\item[{\rm (b) }] $\forall I\in \I(n)$ $g(I)\subset I \subset h(g(I))\cup g(I)$. 

\end{enumerate}
\end{theorem}

\begin{proof}
Note that since $H$ contains a cycle, $m^*_2(H)\geq 1$.
Let $I\in \I(n)$. As in the proof of Theorem~\ref{container-main} we will apply Proposition~\ref{container-main} repeatedly.
Since $H$ is Erd\H{o}s-Simonovits good, we can use 
\[\beta=\max\{k^{-1}n^{-\phi(\al,H)}, k^{-\lam^*(\al,H)}\}\mbox{ and }
\mu=k\beta/\eps=(1/\eps) \max\{n^{-\phi(\al,H)}, k^{1-\lam^*(\al,H)}\}.\]

let $G_0=K_n$. We apply Proposition~\ref{container-prop}, with $G_0$ playing the role of $G$
to obtain the function $f_{G_0}$ and a subset $T_1$ of $I$ with $T_1\subset I \subset f_{G_0}(T_1)$ where
$T_1$ and $f_{G_0}(T_1)$ satisfy the additional properties described in Proposition~\ref{container-prop}.
Now, let $G_1=f_{G_0}(T_1)\setminus T_1$ and $I_1=I\cap G_1=I\setminus T_1$.
Apply Proposition~\ref{container-prop} again, with $G_1$ playing the role of $G$ and $I_1$ playing the role of $I$
to obtain the function $f_{G_1}$ and a subset $T_2$ of $I_1$ with $T_2\subset I_1\subset f_{G_1}(T_2)$.
We continue like this until we arrive at a graph $G_{m(I)}$ with at most $kn^\al$ edges. 

Let $g(I) =  T_1 \cup T_2 \cup \cdots \cup T_{m(I)}$ , where elements of $T_i$ are colored with color $i$.   Let $\cS(s) = \{g(I) :|g(I)| =s \}$.  Let $h(g(I)) = G_{m(I)}$. As in the proof of Theorem~\ref{container-main}, $h$ is well defined.
Furthermore,  as $g(I) \subset I \subset h(g(I)) \cup g(I)$,  conditions (b) is fulfilled.   
It remains to show (a).  We begin by partitioning $\cS(s)$ into sets $\cS_m(s)$ where $\cS_m(s) = \{ S \in \cS(s) : \text{ the edges of } S  \text{ are colored with } m \text{ colors}\}$.

For each $m\in \mathbb{N}$, 

\[\K(m) = \{\bk = (k_1,  \cdots k_m): k_j \in \mathbb{R} ,  (1-\eps)^{j - m}k \leq k_j \leq (1 - \eps)^jn^{2 - \alpha} \text{ and } k_jn^{\alpha} \in \mathbb{N} \}\]

And for each $\bk \in \K(m)$, 

\[\A(\bk) = \{\ba= (a_1,  \dots a_m) : a_j \in \mathbb{N},  a_j \leq \frac{1}{\eps}\max\{k_j^{1 - \lam^*(\al,H)}, n^{-\phi(\al,H)}\}n^{\alpha} \text{ and } \sum_j a_j = s \}\]

So each sequence in $\bk \in \K(m)$ corresponds to a potential sequence $(G_1, \cdots G_m)$ where $e(G_i) = k_jn^{\alpha}$,  as the edges of $(1-\eps)^{j - m}kn^{\alpha} \leq e(G_j) \leq (1 - \eps)^j n^2$ and by Proposition~\ref{container-prop} and Remark~\ref{container-remark}, we have $ e(T_{j})\leq \frac{1}{\eps}\max\{k_j^{1 - \lam^*(\al,H)}, n^{-\phi(\al, H)}\}n^{\alpha}$. Note further that our algorithm returns pairs $(G_i, T_i)$,  such that each sequence of $(T_1,  \cdots T_m)$ is uniquely identified with a sequence $(G_1,  \cdots G_m)$.  Thus it suffices to only count the choices of $T_i$.   The sequence $\ba \in \A(\bk)$  corresponds to a sequence of sizes for $T_i$.  Thus,  we have that for a fixed $m$

\[|\cS_m(s))| \leq \sum_{\bk\in \K(m)} \sum_{\ba\in \A(\bk)} \prod_{j=1}^m \binom{k_j n^\al}{a_j}.\]

Now,  given a $\bk\in \K(m)$ and $\ba\in \A(\bk)$, let us partition the product over $j$ according
to whether $k_j^{1 - \lam^*(\al,H)} < n^{-\phi(\al, H)}$ (call this type 1) or not (call this type 2).  Because $|\K(m)| = 0$ for values of $m \gg \log(n)$ as $(1- \eps)^m n^{2} < k$, and some absolute constant $C_1$ the product over type 1 $j$'s is at most
\[(n^2)^{\sum_j a_j}\leq \expo\left( C_1\cdot n^{\al-\phi(\al, H)} (\log n)^2\right) \leq \expo\left(C_1\cdot k^{1-\lam^*(\al,H)}n^\al\right),\]
where in the last step, we used the fact that $k\leq n^\frac{\phi(H)}{\lam^*(\al,H)-1} (\log n)^{\frac{2}{1-\lam^*(\al,H)}}$.
For each type 2 $j$,  we have $k_j^{1 - \lam^*(\al,H)} \geq n^{-\phi(\al, H)}$ and thus $a_j\leq \frac{1}{\eps} k_j^{1-\lam^*(\al,H)}n^\al$.
From this we get
\[k_j\leq (1/\eps)^{\frac{1}{\lam^*(\al,H)-1} } n/ a_j^{\frac{1}{\lam^*(\al,H)-1}}.\]
Hence

\[\binom{k_jn^\al}{a_j}\leq \left(\frac{ek_jn^\al}{a_j}\right)^{a_j}  \leq \left( \frac{n^\al}{\eps a_j } \right )^{\frac{\lam^*(\al,H)}{\lam^*(\al,H)-1}\cdot a_j}\]

Applying Lemma~\ref{sum-estimate} with $M=k^{1-\lam^*(\al,H)} n^\al,1 - \delta=(1 - \eps)^{\lam^*(\al, H) - 1}$, the product over type 2 $j$'s is at most
\[\left ( \frac{C_2n^\al}{s}\right) ^{\frac{\lam^*(\al, H)}{\lam^*(\al, H)-1}\cdot s}\cdot \expo\left ( C_2k^{1-\lam^*(\al,H)} n^\al\right),\]
for some $C_2=C(H)$.  

Thus, letting $C' = 2\max\{C_1, C_2\}$,  \[|\cS_m(s)| \leq \sum_{\bk\in \K(m)} \sum_{\ba\in \A(\bk)}\left( \frac{C'n^\al}{s}\right) ^{\frac{\lam^*(\al,H)}{\lam^*(\al,H)-1}\cdot s}\cdot \expo\left( C'k^{1-\lam^*(\al,H)} n^\al\right)\]

And thus,  \[|\cS(s)| \leq \sum_{m = 1}^{\infty} \sum_{\bk\in \K(m)} \sum_{\ba\in \A(\bk)}\left( \frac{C'n^\al}{s}\right) ^{\frac{\lam^*(\al,H)}{\lam^*(\al,H)-1}\cdot s}\cdot \expo\left( C'k^{1-\lam^*(\al,H)} n^\al\right)\]

As in the proof of Theorem~\ref{container-main}, we have
$\sum_{m=1}^\infty \sum_{\bk\in \K(m)} |\A(\bk)|=n^{O(\log n)}$. The theorem thus follows.
\end{proof}

\begin{theorem}[Restatement of Theorem~\ref{random-Turan1}]
Let $H$ be a bipartite graph with $h$ vertices and $\ell$ edges that contain a cycle. Let $A,\al$ be positive reals such that $\ex(n,H)\leq An^\al$
for every $n\in \mathbb{N}$.  There exists a constant $C=C(\al, H)$ such that 
\[\ex(G(n,p), H)\leq \begin{cases} 
   Cn^{2-\frac{1}{m_2(H)}} & \text{if } p\leq n^{-\frac{1}{m_2(H)}},\\
Cp^{1-\frac{1}{\lam(\al,H)} }n^\al & \text{otherwise } 
\end{cases}\]
with high probability as $n\to \infty$.
\end{theorem}
\begin{proof}
Since $\ex(G(n,p),H)$ is an increasing function of $e(G)$, it suffices to prove the claimed bound in the case
$p\geq n^{-\frac{1}{m_2(H)}}$.  
Given such a function $p=p(n)$, define $k=p^{-\frac{1}{\lam(\al, H)}}$.  Note that  $k\leq n^{2-\al}$.   
Suppose that there exists an $H$-free subgraph $I\subset G(n,p)$ with $m$ edges. Then by Theorem~\ref{container-main} with our choice of $k$ there exist functions $g,  h$ on the set of independent sets, with the property $g(I)\subset I \subset h(g(I))\cup g(I)$. Therefore, we know that   $g(I) \subset G(n,p)$ and $G(n,p)$ has at least $m - e(g(I))$ edges of $h(g(I))$.  The probability of this event
is at most
\[\sum_{S\in\cS}\binom{kn^\al}{m-e(S)}\cdot p^m \leq 
\sum_{s=0}^{C'k^{1-\lam(\al,H)} n^\al }\left (\frac{C'p^{\frac{\lam(\al,H)-1}{\lam(\al,H)}} n^\al}{s} \right) ^{\frac{\lam(\al,H)}{\lam(\al,H)-1} \cdot s}\cdot 
\expo\left(C'k^{1-\lam(\al,H)} n^\al\right) \cdot \left( \frac{3pkn^\al}{m-s} \right) ^{m-s}\]
\[\leq \expo\left( O(1)\cdot (p^{\frac{\lam(\al,H)-1}{\lam(\al,H)}} n^\al + k^{1-\lam(\al,H)} n^\al ) \right) \cdot \left( \frac{4pkn^\al}{m}\right)^{m/2}
\to 0,\]
as $n\to\infty$, as long as 
\[m\geq C \cdot  \max\{p^{\frac{\lam(\al,H)-1}{\lam(\al,H)}} n^\al, k^{1-\lam(\al,H)} n^\al , pkn^\al\}
= C p^{\frac{\lam(\al,H)-1}{\lam(\al,H)}} n^\al,
\]
for some sufficiently large constant $C=C(\al, H)$, where we used the definition of  $k=p^{-\frac{1}{\lam(\al,H)}}$.
%Note that $m-s\geq m/2$ under this assumption.
%Then this implies 
%\[m = C p^{\frac{\lam(\al,H)-1}{\lam(\al,H)}} n^\al \mbox{ while }  p \geq n^{-\frac{1}{m_2(H)}}.\]
%The theorem follows.
\end{proof}

\begin{theorem} [Restatement of Theorem~\ref{random-Turan2}]    \label{turanES}
Let $H$ be a bipartite graph with $h$ vertices and $\ell$ edges such that $H$ contains a cycle and
is Erd\H{o}s-Simonovits good. Let $A,\al$ be positive reals such that $\ex(n,H)\leq An^\al$
for every $n\in \mathbb{N}$.  There exists a constant $C=C(\al,  H)$ such that 
\[\ex(G(n,p), H)\leq \begin{cases}  Cn^{\alpha - \phi(\al, H)}(\log(n))^{2} & \text{if } p\leq n^{-\frac{\phi(\al, H)\lam^*(\al,H)}{\lam^*(\al, H)-1} } \cdot (\log n)^{\frac{2\lam^*(\al, H)}{\lam^*(\al, H)-1}} 
\\
Cp^{1-\frac{1}{\lam^*(\al, H)} }n^\al & \text{otherwise}
\end{cases}\]
with high probability as $n\to \infty$.
\end{theorem}
\begin{proof}
Since $\ex(G(n,p),H)$ is an increasing function of $e(G)$, it suffices to prove the claimed bound in the case
$p\geq n^{-\frac{\phi(H)\lam^*(\al,H)}{\lam^*(\al,H)-1} } \cdot (\log n)^{\frac{2\lam^*(\al,H)}{\lam^*(\al, H)-1}}$.
Let $k = p^{-\frac{1}{\lam^*(\al,H)}}$. Note that by our choice of $p$,
\[k\leq n^{\frac{\phi(\al,H)}{\lam^*(\al, H)-1}} / (\log n)^{\frac{2}{\lam^*(\al, H)-1}}.\]
Suppose that there exists an $H$-free subgraph $I \subset G(n,p)$ with $m$ edges. Then by Theorem~\ref{container-ES} with our choice of $k$ there exist functions $g,  h$ such that  $g(I) \subset G(n,p)$ and $G(n,p)$ has at least $m - e(g(I))$ edges of $h(g(I))$.  The probability of this event
is at most
\[\sum_{S\in\cS}\binom{kn^\al}{m-e(S)}\cdot p^m \leq 
\sum_{s=0}^{C'k^{1-\lam^*(\al,H)} n^\al} \left (\frac{C'p^{\frac{\lam^*(\al,H)-1}{\lam^*(\al,H)}} n^\al}{s} \right) ^{\frac{\lam^*(\al,H)}{\lam^*(\al,H)-1} \cdot s}\cdot 
\expo\left(C'k^{1-\lam^*(\al,H)} n^\al\right) \cdot \left( \frac{3pkn^\al}{m-s} \right) ^{m-s}\]
\[\leq \expo\left( O(1)\cdot (p^{\frac{\lam^*(\al,H)-1}{\lam^*(\al, H)}} n^\al + k^{1-\lam^*(\al,H)} n^\al ) \right) \cdot \left( \frac{4pkn^\al}{m}\right)^{m/2}
\to 0,\]
as $n\to\infty$, as long as 
\[m\geq C \cdot \max\{p^{\frac{\lam^*(\al,H)-1}{\lam^*(\al,H)}} n^\al, k^{1-\lam^*(\al,H)} n^\al , pkn^\al\} = Cp^{\frac{\lam^*(\al,H)-1}{\lam^*(\al,H)}} n^\al\] for some sufficiently large  $C(\al, H)$, where we used the definition of $k = p^{-\frac{1}{\lambda*(\al, H)}}$. 
%One can check that these inequalities hold if 
%\[p\geq n^{-\frac{\phi(H)\lam^*(\al,H)}{\lam^*(\al,H)-1} } \cdot (\log n)^{\frac{2\lam^*(\al,H)}{\lam^*(\al,H)-1}} \mbox{ and }
%m\gg p^{\frac{\lam^*(\al,H)-1}{\lam^*(\al,H)}} n^\al.\]
%The theorem follows.
\end{proof}

As mentioned earlier, Theorem~\ref{turanES} implies Morris and Saxton's result on $C_{2\ell}$.

\begin{corollary}[\cite{MS}] \label{c2ell-random}
For every $\ell\geq 2$, there exists a constant $C=C(\ell)>0$ such that

\[\ex(G(n,p), C_{2\ell})\leq \begin{cases}  Cn^{1 + 1/(2\ell - 1)}(\log(n))^{2} & \text{if } p\leq n^{- (\ell - 1)/(2\ell-1)}  \cdot (\log n)^{2\ell} 
\\
Cp^{1/\ell } n^{1 + 1/\ell} & \text{otherwise}
\end{cases}\]

with high probability as $n\to \infty$.
\end{corollary}
\begin{proof}
It is known that $C_{2\ell}$ is Erd\H{o}s-Simonovits good with $\al=1+1/\ell$.
Apply Theorem~\ref{turanES} with $H = \C_{2\ell}$ and $\al=1+1/\ell$,
 noting that $\lambda^*(\al,C_{2\ell}) = \frac{\ell}{\ell -1}, \phi(\al, H)= \frac{\ell -1}{\ell(2\ell-1)}$. The corollary then follows directly. 
\end{proof}

%%\newpage %% AUTHOR: please comment out this line.  It serves only
%%   to demonstrate both types of header line in aic-template.pdf

%%\section{Expansion estimates}

%% More of the body of your paper goes here~\cite{bergelson-johnson-moreira}.

%%% AUTHOR: optional appendix here
%%\appendix %% you may comment this out if no Appendix
%%\section*{Appendix}
%%\section{Improving the constants}
%%Material is placed here as needed.

%%% AUTHOR: optional acknowledgments here
%%\section*{Acknowledgments} %%  you may comment this out if no Ackno
%%The authors are grateful to the anonymous reviewers for finding a bug in the main result.

%%% AUTHOR:
%%% Bibliography goes here. Note that the arXiv cannot process bibtex
%%% or biber bibliographies.  Example of acceptable bibliograpy format:
\bibliographystyle{amsplain}

%% AUTHOR: You can generate such a bibliography from a .bib file by 
%% running pdflatex/bibtex/pdflatex/pdflatex and then pasting the .bbl file
%% between \begin{thebibliography} and \end{bibliography}

%%% AUTHOR: Include a short description of each author following the
%%% structure below. Use the same short tags used previously.  
%%% Use \imageat{} and \imagedot{} instead of "@" and "." in
%%% email addresses-this replaces the symbols with graphics to avoid 
%%% e-mail address harvesting from the .pdf file
\begin{aicauthors}
\begin{authorinfo}[tjiang]
          Tao Jiang \\
  Department of Mathematics\\
  Miami University\\
  Oxford, OH 45056, USA.\\
  jiangt\imageat{}miamioh\imagedot{}edu 
\end{authorinfo}
\begin{authorinfo}[slong]
  Sean Longbrake\\
   Department of Mathematics\\
  Miami University\\
  Oxford, OH 45056, USA.\\
  Currently At: \\
  Department of Mathematics\\ Emory University \\ Atlanta, GA 30322, USA. \\
  sean\imagedot{}longbrake\imageat{}emory\imagedot{}edu
\end{authorinfo}
\end{aicauthors}

\end{document}